\documentclass[11pt, a4paper,leqno]{amsart}
\usepackage{amsmath,amsthm,amscd,amssymb,amsfonts, amsbsy}
\usepackage{latexsym}
\usepackage{txfonts}
\usepackage{exscale}

\usepackage[colorlinks,citecolor=red,pagebackref,hypertexnames=false]{hyperref}
\usepackage{color}

\calclayout
\allowdisplaybreaks

\theoremstyle{plain}
\newtheorem{theorem}[equation]{Theorem}
\newtheorem{lemma}[equation]{Lemma}

\newtheorem{proposition}[equation]{Proposition}

\theoremstyle{definition}
\newtheorem{definition}[equation]{Definition}

\theoremstyle{remark}
\newtheorem{remark}[equation]{Remark}

\newtheorem{claim}[equation]{Claim}

\numberwithin{equation}{section}

\newcommand{\dist}{\operatorname{dist}}

\newcommand{\re}{\mathbb{R}}
\newcommand{\rn}{\mathbb{R}^n}

\newcommand{\ree}{\mathbb{R}^{n+1}}

\newcommand{\dd}{\mathbb{D}}

\newcommand{\om}{\Omega}

\renewcommand{\H}{\mathcal{H}}

\newcommand{\W}{\mathcal{W}}
\newcommand{\B}{\mathcal{B}}

\newcommand{\sbf}{{\bf S}}

\newcommand{\G}{\mathcal{G}}

\newcommand{\pom}{\partial\Omega}

\usepackage{textcomp}

\renewcommand{\emptyset}{\mbox{\textup{\O}}}

\DeclareMathOperator{\diam}{diam}

\def\div{\mathop{\operatorname{div}}\nolimits}

\begin{document}
\allowdisplaybreaks

\title[Parabolic CME and Parabolic UR] {Carleson measure estimates for caloric functions and parabolic uniformly rectifiable sets}
\author{S. Bortz}
\address{Department of Mathematics
\\
University of Alabama
\\
Tuscaloosa, AL, 35487, USA}
\email{sbortz@ua.edu}
\author{J. Hoffman}
\address{Department of Mathematics
\\
University of Missouri
\\
Columbia, MO 65211, USA}
\email{jlh82b@mail.missouri.edu}
\author{S. Hofmann}
\address{
Department of Mathematics
\\
University of Missouri
\\
Columbia, MO 65211, USA}
\email{hofmanns@missouri.edu}
\author{J-L. Luna Garcia}
\address{Department of Mathematics
\\
University of Missouri
\\
Columbia, MO 65211, USA}
\email{jlwwc@mail.missouri.edu}
\author{K. Nystr\"om}
\address{Department of Mathematics, Uppsala University, S-751 06 Uppsala, Sweden}
\email{kaj.nystrom@math.uu.se}

\thanks{The authors J. H., S. H., and J-L. L. G.  were partially supported by NSF grants
DMS-1664047 and DMS-2000048. K.N was partially supported by grant  2017-03805 from the Swedish research council (VR)}

\begin{abstract}
Let  $E \subset \mathbb R^{n+1}$ be a parabolic uniformly rectifiable set. We prove that every bounded solution $u$ to
$$\mbox{$\partial_tu- \Delta u=0$ in $\mathbb R^{n+1}\setminus E$,}$$ satisfies a Carleson measure estimate condition. An important
technical novelty of our work is that we develop a corona domain approximation scheme for $E$  in terms of {regular} Lip(1/2,1) graph domains. This approximation scheme has an analogous elliptic version which is an improvement of the known results in that setting.
\end{abstract}

\maketitle
\tableofcontents

\section{Introduction}

For more than forty years, there has been significant interest in quantitative estimates for solutions of (linear) elliptic and parabolic partial differential equations in the absence of smoothness. In this area of research, the lack of
smoothness presents itself in the structure or
regularity of the coefficients of the operator,
or in the geometry of the domain. Recently, sustained efforts in this area have provided characterizations of quantitative geometric notions (e.g. uniform rectifiability) in terms of quantitative estimates for harmonic functions \cite{GMT, HMM1} and a geometric characterization of the $L^p$-solvability of the Dirichlet problem \cite{AHMMT}.   This paper concerns the parabolic analog of \cite{HMM1} and overcomes the substantial difficulty
introduced by the distinguished time direction and the anisotropic scaling.
To deal with this difficulty,
we are forced to build appropriate approximating domains with better properties
than would be enjoyed by the parabolic analogues of the chord-arc domains constructed in \cite{HMM1}.
In particular, our construction improves on that of \cite{HMM1}, even in the elliptic setting.
We shall discuss these issues in more detail momentarily.

We shall prove the following.

\begin{theorem}[A Carleson Measure Estimate for Bounded Caloric Functions]\label{PURCMEMT.thrm}
Let $n \ge 2$. Let $E \subset \ree$ be a set which is uniformly rectifiable in the parabolic sense. Then for any solution to $(\partial_t - \Delta_X)u = 0$ in $\ree \setminus E$ with $u \in L^\infty(\ree \setminus E)$ it holds that
\begin{equation}\label{PURCMEMTeq.eq}
\sup_{(t,X) \in E, r > 0} r^{-n-1}
\iint\limits_{B((t,X),r)} |\nabla u|^2 \delta(s,Y) \,  dY \, ds \le C \|u\|_{L^\infty(E^c)}^2,
\end{equation}
where $\delta(s,Y)  :=\dist((s,Y), E)$, and $C$ depends only
dimension and the parabolic UR constants for $E$.
\end{theorem}

Here, and below,
$\dist((s,Y), E)$ is the parabolic distance from $(s,Y)$ to the given set $E$,  and the ball $B((t,X),r)$
is defined with respect to the parabolic metric; see \eqref{paradist} and \eqref{paraball} below.

In the case that $\om$ is an open set, the following holds.
\begin{theorem}\label{PURCMEMT2.thrm}
Let $n \ge 2$.  Let $\om \subset \ree$ be an open set for which $\pom$ is uniformly rectifiable in the parabolic sense. Then for any solution to $(\partial_t - \Delta_X)u = 0$ in $\om$ with $u \in L^\infty(\om)$ it holds that
\[\sup_{(t,X) \in E, r > 0} r^{-n-1}
\iint\limits_{B((t,X),r) \cap \om} |\nabla u|^2 \delta(s,Y) \,  dY \, ds \le C \|u\|_{L^\infty(\om)}^2,\]
where $C$ depends only dimension and the parabolic UR constants for $\pom$. Here,
$\delta(s,Y)  := \dist((s,Y), \pom)$ the parabolic distance to $\pom$.
\end{theorem}

The notion of parabolic uniform rectifiability was introduced in \cite{HLN-BP,HLN1} and is defined below, but we provide some context here. In a series of works J. Lewis and his collaborators \cite{H1,H,HL,HL-JFA, LM, LS}, showed that the `good' parabolic graphs for parabolic singular integrals and parabolic potential theory are regular Lip(1/2,1) graphs, that is, graphs which are Lip(1/2,1) (in time-space coordinates) and which possess extra regularity in time in the sense that a (non-local) half-order time-derivative of the defining function of the graph is in the space of functions of parabolic bounded mean oscillation. This is in contrast to the elliptic setting, where one often
views Lipschitz graphs as the `good' graphs for singular integrals and potential theory (because of the works \cite{CMM, CDM, CS, Dahlberg, D4}), and where
the $BMO$ estimate for the gradient is an automatic
consequence of Rademacher's theorem and the inclusion of $L^\infty$ in $BMO$.  The definition of parabolic uniform rectifiability in \cite{HLN-BP,HLN1} is given in terms of parabolic ``$\beta$ numbers\footnote{These $\beta$ numbers can be traced back to the work of P. Jones \cite{Jones}.}"  but in this paper we do not work with this definition directly. Instead, we here work with an equivalent notion of parabolic uniform rectifiability in terms of the existence of appropriate corona decompositions recently established in \cite{BHHLN1, BHHLN2}. However, it is worth remarking that the graph of a Lip(1/2,1) function is parabolic uniformly rectifiable if and only if the function has a half-order time derivative in parabolic BMO. In contrast to the case of `elliptic' uniform rectifiablity, which has reached a state of maturity that includes numerous interesting characterizations, this is not the case for parabolic uniform rectifiablity. In fact, beyond \cite{HLN-BP,HLN1} the only correct and more
systematic studies of parabolic uniformly rectifiable sets can be found in \cite{BHHLN1, BHHLN2}\footnote{There are works of J. Rivera-Noriega in this area, but these articles have significant gaps or no proofs. Some of these gaps are outlined in \cite{BHHLN1, BHHLN2, BHHLN3}.}.   In \cite{BHHLN1, BHHLN2} parabolic uniform rectifiability is characterized in terms of a bilateral coronaization by regular Lip(1/2,1) graphs (Lemma \ref{bilateralcorona.lem}), and this characterization is the starting point for the analysis in this paper. In general there are many interesting open problems in this and related areas, and it should be emphasized that parabolic uniform rectifiability is {\it significantly} different to
its `elliptic' counterpart, see \cite[Observation 4.19]{BHHLN1}.

To give an idea of the methods involved in the proof of Theorem \ref{PURCMEMT.thrm}, the primary novelty of our work is a `corona domain approximation scheme' (Proposition \ref{stopdomainconst.prop}) in terms
of {regular} Lip(1/2,1) graph {\it domains}. This is in contrast to the (elliptic) NTA domain
approximations produced in \cite{HMM1}  for uniformly rectifiable sets. In fact, our proof here carries over without modification\footnote{Except that the technical Lemma \ref{regL121.lem} is no longer needed.} to the elliptic setting, providing an (improved) approximation by Lipschitz domains. In \cite{HMM1} the authors use Whitney cubes to construct these NTA domains using `dyadic sawtooths' and exploiting an elliptic bilateral corona decomposition. The heuristic in the elliptic setting is that these `sawtooth domains' inherit many essential properties of the original
 boundary. In contrast, in the parabolic setting the analogous constructions do not
necessarily inherit even the most basic properties. One of the most
readily apparent difficulties in the parabolic setting comes from the fact that the natural `lower dimensional parabolic measure' can easily fail to see relatively nice sets. In particular, given a cube (with respect to the standard coordinates) in $\ree$,
two of the faces (those orthogonal to the time axis)
have zero natural `parabolic surface measure', which says that, not only does the boundary of a cube fail to be uniformly rectifiable in the parabolic sense, it fails even to have the Ahlfors-David regularity property. The method outlined in this paper circumvents this difficulty by `lifting' the graphs in the parabolic bilateral corona decomposition (Lemma \ref{bilateralcorona.lem}) in a manner that respects the stopping time regimes and thereby produces the graph domains rather directly. We also point out that, while the analogous elliptic results (in \cite{HMM1}) proceed along the lines of `extrapolation of Carleson measures' it was later seen in \cite{HMM2} that this was unnecessary and a more direct approach is available. Therefore, upon proving Proposition \ref{stopdomainconst.prop}, the proof of Theorem \ref{PURCMEMT.thrm} proceeds as in \cite{HMM2}.

Let us provide some motivation for the estimate in Theorem \ref{PURCMEMT.thrm}. As remarked above, (elliptic) uniform rectifiability has been characterized by various properties of harmonic functions and among these characterizations is the elliptic version of the Carleson measure estimate in Theorem \ref{PURCMEMT.thrm} (see \cite{GMT, HMM1}).  We therefore expect that
the estimate in Theorem \ref{PURCMEMT.thrm} is a significant step in characterizing parabolic uniform rectifiability by properties of caloric functions.  We suspect that additional considerations/conditions will need to be made, as was the case for non-symmetric operators in the elliptic setting \cite{AGMT}, in the converse, `free-boundary' direction, due to the lack of self adjointness of the heat operator.
In domains that are sufficiently nice topologically,
the estimate in Theorem \ref{PURCMEMT.thrm} (and its elliptic analogue) is also intimately
tied to the solvability of the $L^p$-Dirichlet boundary value problem in the parabolic setting \cite{DPP,GH} (see \cite{KKoPT,KKiPT}, and  related
work in \cite{DKP, HLe, Zhao} for the elliptic theory). Indeed, in the case of
regular Lip(1/2,1) graph domains it is known that estimate \eqref{PURCMEMTeq.eq} for bounded null-solutions to
general parabolic operators of the form ${\mathcal{L}} = \partial_t - \div_X A \nabla_X$ is equivalent to the solvability
of the $L^p$-Dirchlet boundary value problem for some $p> 1$ \cite{DPP} (boundary value problem
means the data is prescribed on the lateral boundary). In fact, merely\footnote{In particular,
without assuming
that the domain is the region above a regular Lip(1/2,1) graph.} assuming
parabolic Ahlfors-David regularity and a `backwards thickness condition' (also
of Ahlfors-David regular type),
the solvability of the $L^p$-Dirichlet boundary value problem is implied by a
stronger estimate where the $L^\infty$ norm on the right hand side of \eqref{PURCMEMTeq.eq} is replaced by the (boundary) $BMO$ norm of the data, see \cite{GH}. This stronger estimate is unlikely\footnote{The elliptic analogue does not hold (in general) in the complement of uniformly rectifiable set.} to hold in the present setting due to the lack of (non-tangential) accessibility to the boundary.

The rest of this paper is organized as follows. In Section \ref{prelim.sect} we introduce the notions and notation used throughout the paper. In Section \ref{domainapprox.sect} we construct approximating domains, each adapted to a particular stopping time regime in the parabolic bilateral corona decomposition, Lemma \ref{bilateralcorona.lem}. In Section \ref{proofmt.sect} we prove the main theorems of the paper (Theorems \ref{PURCMEMT.thrm} and \ref{PURCMEMT2.thrm}) using the constructions produced in Section \ref{domainapprox.sect}. In Section \ref{FR.sect} we discuss some possible extensions of the results here.

\section{Preliminaries}\label{prelim.sect}
Throughout this paper, we work in $\ree$ identified\footnote{We apologize for the departure from the usual notation $(X,t)$, but we will often be working with graphs and it is convenient to have the `last' variable as the graph variable.} with $\re \times \rn = \{(t,X): t \in \re, X \in \rn\}$ and $n \in \mathbb{N}$, $n \ge 2$. We use the notation
\begin{equation}\label{paradist}
\dist(A,B) := \inf_{(t,X) \in A, (s,Y) \in B} |X -Y| + |t-s|^{1/2},
\end{equation}
to denote the parabolic distance between  $A$ and $B$, $A,B \subseteq \ree$. We also use the notation $B((t,X),r)$ for the parabolic ball centered at $(t,X)$ with radius $r > 0$, that is,
\begin{equation}\label{paraball}
B((t,X),r)  := \{(s,Y): \dist((t,X), (s,Y)) < r\}.
\end{equation}
Given a $E \subset \mathbb{R}^{n+1}$ we let $\diam(E)$ denote its the diameter, or parabolic diameter, defined with respect to the parabolic metric.
\begin{definition}[Parabolic Hausdorff measure]
Given $s > 0$ we let $\H^s_p$ denote the $s$-dimensional parabolic Hausdorff measure. More specifically, for a set $E \subset \mathbb{R}^{n+1}$ and $\epsilon > 0$ we define
\[\H^s_{p,\epsilon}(E) := \inf\left\{\sum_i \diam(E_i)^s: E \subseteq \cup_i E_i, \diam(E_i) \le \epsilon \right\},\]
and
\[\H^s_{p}(E) := \lim_{\epsilon \to 0^+} \H^s_{p,\epsilon}(E) = \limsup_{\epsilon \to 0^+} \H^s_{p,\epsilon}(E).\]
\end{definition}

The following family of planes will be important in this work.
\begin{definition}[$t$-independent planes]
We say that an $n$-dimensional plane $P$ in $\ree$ is $t$-independent if it contains a line in the $t$-direction. Equivalently, if $\vec{\nu}$ is the normal vector to $P$ then $\vec{\nu} \cdot (1,\vec{0}) = 0$.
\end{definition}

The following local energy (Caccioppoli) inequality holds for solutions to the heat equation.
\begin{lemma}[Caccioppoli Inequality]\label{cacc}
Let $B = B((t,X),r)$ and suppose that $u$ is a solution to $(\partial_t - \Delta_X)u = 0$ in $(1 + \alpha)B$, for some $\alpha > 0$ then
\[\int_{B} |\nabla_X u(t,X)|^2 \, dX\, dt \lesssim r^{-2} \int_{(1+ \alpha)B} |u|^2 \, dX\, dt, \]
where the implicit constant depends on dimension and $\alpha$.
\end{lemma}

\begin{definition}[Ahlfors-David Regular]\label{def1.ADR}
We say a $E \subset \ree$ is (parabolic) Ahlfors-David regular, written $E$ is ADR, if it is closed and there exists a constant $C > 0$ such that
\[C^{-1} r^{n+1} \le \H^{n+1}_p(B((t,X),r) \cap E) \le C r^{n+1}, \quad \forall (t,X) \in E, r \in (0, \diam(E)).\]
\end{definition}

We will call the $C$ of Definition \ref{def1.ADR} the Ahlfors-David regularity constant and if a particular constant depends on the Ahlfors-David regularity constant, we will say {that the} constant `depends on ADR'. We will sometimes write $\sigma: = \H^{n+1}_p|_{E}$, to denote the `surface measure' on $E$. (The underlying set defining $\sigma$ will always be clear from context.)

An ADR set $E$ can be viewed as a space of homogeneous type $(E,\dist, \sigma)$, with homogeneous dimension $n+1$. All such sets have a nice filtration, which we will refer to as the `dyadic cubes' on $E$.

\begin{lemma}[\cite{Ch,DS1,DS2, hyt-k}]\label{cubes}  Assume that $E  \subset \mathbb R^{n+1}$ is (parabolic) ADR  in the sense of Definition \ref{def1.ADR} with constant $C$. Then $E$ admits a parabolic dyadic decomposition in the sense that there exist
constants $a_0>0,\, \gamma>0$ and $c_*<\infty$,  such that for each $k \in \mathbb{Z}$
there exists a collection of Borel sets, $\mathbb{D}_k$,  which we will call (dyadic) cubes, such that
$$
\mathbb{D}_k:=\{Q_{j}^k\subset E: j\in \mathfrak{I}_k\},$$ where
$\mathfrak{I}_k$ denotes some (countable)  index set depending on $k$, satisfying
\begin{eqnarray*}
(i)&&\mbox{$E=\cup_{j}Q_{j}^k\,\,$ for each
$k\in{\mathbb Z}$.}\notag\\
(ii)&&\mbox{If $m\geq k$ then either $Q_{i}^{m}\subset Q_{j}^{k}$ or
$Q_{i}^{m}\cap Q_{j}^{k}=\emptyset$.}\notag\\
(iii)&&\mbox{For each $(j,k)$ and each $m<k$, there is a unique
$i$ such that $Q_{j}^k\subset Q_{i}^m$.}\notag\\
(iv)&&\mbox{$\diam\big(Q_{j}^k\big)\leq c_* 2^{-k}$.}\notag\\
(v)&&\mbox{Each $Q_{j}^k$ contains $E\cap B((t^k_j,Z^k_{j}), a_02^{-k})$ for some $(t^k_j,Z^k_{j})\in E$.}\notag\\
(vi)&&\mbox{$E(\{(t,Z)\in Q^k_j: \dist((t,Z),E\setminus Q^k_j)\leq \varrho \,2^{-k}\big\})\leq
c_*\,\varrho^\gamma\,E(Q^k_j),$}\notag\\
&&\mbox{for all $k,j$ and for all $\varrho\in (0,\alpha)$.}
\end{eqnarray*}
\end{lemma}

\begin{remark} We denote by  $\mathbb{D}=\mathbb{D}(E)$ the collection of all $Q^k_j$, i.e., $$\mathbb{D} := \cup_{k} \mathbb{D}_k.$$ Given a cube $Q\in \mathbb{D}$, we set
 \begin{equation*}\label{eq2.discretecarl}
\mathbb{D}_Q:= \left\{Q'\in \mathbb{D}: Q'\subseteq Q\right\}.
\end{equation*}
For a dyadic cube $Q\in \mathbb{D}_k$, we let $\ell(Q) := 2^{-k}$, and we will refer to this quantity as the size or side-length
of $Q$.  Evidently, $\ell(Q)\sim\diam(Q)$ with constant of comparison depending at most on $n$ and $C$.
Note that $(iv)$ and $(v)$ of Lemma \ref{cubes} imply that for each cube $Q\in\mathbb{D}_k$,
there is a point $(t_Q,X_Q) \in E$, and  a ball $B((t_Q,X_Q), r)$ such that
$r\approx2^{-k} \approx\diam (Q)$
and \begin{equation}\label{cube-ball}
E\cap B((t_Q,X_Q),r) \subset Q \subset E\cap B((t_Q,X_Q),Cr) \end{equation}
for some uniform constant $C$. We shall refer to the point $(t_Q,X_Q)$ as the center of $Q$. Given a dyadic cube $Q\subset E$ and $K>1$, we define the $K$ `dilate' of $Q$  by
\begin{equation}\label{dilatecube}
K Q:= \{(t,X) \in E: \dist((t,X),E) < (K-1)\diam(Q)\}.
\end{equation}
\end{remark}

Throughout the paper we assume that $E$ is uniformly rectifiable in the parabolic sense.
 We nominally define this notion in language that will be meaningful to
those intimately familiar with the work of David and Semmes,
but we will here not discuss and introduce all the relevant terminology
 (the interested reader may consult  \cite[Definition 4.]{BHHLN1}),  as it will not be used in the present work.
In fact, the reader can safely ignore the following definition, as parabolic uniform rectifiability is equivalent to the existence of a bilateral corona decomposition \cite[Theorem 3.3]{BHHLN2}
(see Lemma \ref{bilateralcorona.lem} below)  and the latter is the formulation of parabolic uniform rectifiability that we will actually use throughout the paper.

\begin{definition}[Uniformly Rectifiable in the Parabolic sense (P-UR)]
We say a set $E \subset \ree$ is uniformly rectifiable in the parabolic sense (P-UR) if $E$ is ADR and satisfies the $(2,2)$ geometric lemma with respect to $t$-independent planes and the measure $\H^{n+1}_p$. See \cite[Definition 4.1]{BHHLN1}\footnote{In \cite{BHHLN1}, a different measure was used in place of $\H^{n+1}_p$, but these measures are equivalent when the set $E$ is P-UR (with respect to either measure). See \cite[Corollary B.2]{BHHLN2}. }. We say that a constant depends on P-UR if it depends on the ADR and Carleson measure constant in the definition of the $(2,2)$ geometric lemma (with respect to $t$-independent planes and the measure $\H^{n+1}_p$).
\end{definition}

In order to state the bilateral corona decomposition, we need to define regular Lip(1/2,1) graphs and coherent subsets of dyadic cubes.

\begin{definition}[Regular Lip(1/2,1) graphs]
We say that $\Gamma$ is a regular Lip(1/2,1) graph if there exists a $t$-independent plane $P$ and a function $\psi: P \to P^\perp$ such that,
\[\Gamma = \{(p, \psi(p)): p \in P\},\]
where, upon identifying $P$ with $\rn = \re \times \re^{n-1} = \{(t,x'): t \in \re, x' \in \re^{n-1}\}$, there exists two constants $b_1,b_2$ such that $\psi$ has two properties:
\begin{itemize}
\item $\psi$ is a Lip(1/2,1) function with constant bounded by $b_1$, that is
\[|\psi(t,x') - \psi(s,y')| \le b_1(|x' - y'| + |t -s|^{1/2}), \quad \forall (t,x'), (s,y') \in \rn.\]
\item $\psi$ has a half-order time derivative in parabolic-$BMO$ with parabolic-$BMO$ norm bounded by $b_2$, that is,
\[\|D_t^{1/2} \psi\|_{P\text{-}BMO(\rn)} \le b_2, \]
where ${P\text{-}BMO}$ is the space of bounded mean oscillation with respect to parabolic balls (or cubes) and $D_t^{1/2} \psi(t,x')$ denotes the half-order time derivative. The half-order time derivative of $\psi$ can be defined by the Fourier transform or by
\[D_t^{1/2} \psi(t,x'):= \hat{c} \text{ p.v.}\int_{\re} \frac{\psi(s,x') -\psi(t,x') }{|s-t|^{3/2}} \, dt, \quad \forall t \in \re, \forall x' \in \re^{n-1},\]
where $\hat{c}$ is an appropriate constant.
\end{itemize}
\end{definition}

\begin{definition}[Coherency \cite{DS2}]\label{d3.11}
Suppose $E$ is a $d$-dimensional ADR set with dyadic cubes $\dd(E)$. Let $\sbf\subset \dd(E)$. We say that $\sbf$ is
``coherent" if the following conditions hold:
\begin{itemize}\itemsep=0.1cm

\item[$(a)$] $\sbf$ contains a unique maximal element $Q(\sbf)$ which contains all other elements of $\sbf$ as subsets.

\item[$(b)$] If $Q$  belongs to $\sbf$, and if $Q\subset \widetilde{Q}\subset Q(\sbf)$, then $\widetilde{Q}\in {\bf S}$.

\item[$(c)$] Given a cube $Q\in \sbf$, either all of its children belong to $\sbf$, or none of them do.

\end{itemize}
We say that $\sbf$ is ``semi-coherent'' if only conditions $(a)$ and $(b)$ hold.
\end{definition}

The following is the bilateral corona decomposition.
\begin{lemma}[{\cite[Theorem 3.3]{BHHLN2}}]\label{bilateralcorona.lem} Suppose that $E\subset \ree$ is P-UR.
Given any positive constant
$\eta \ll 1$ and  $K := \eta^{-1}$, there are constants $C_{\eta}= C_\eta(\eta, n, ADR, \text{P-UR})$ and $b_2 = b_2(n,ADR, \text{P-UR})$ and a disjoint decomposition
$\dd(E) = \G\cup\B$, satisfying the following properties.
\begin{enumerate}
\item  The ``Good"collection $\G$ is further subdivided into
disjoint `stopping time regimes', $\G = \cup_{\sbf^* \in \mathcal{S}} \sbf^*$ such that each such regime $\sbf^*$ is coherent.
\item The ``Bad" cubes, as well as the maximal cubes $Q({\sbf^*})$ satisfy a Carleson
packing condition:
$$\sum_{Q'\subset Q, \,Q'\in\B} \sigma(Q')
\,\,+\,\sum_{{\sbf^*}: Q({\sbf^*})\subset Q}\sigma\big(Q({\sbf^*})\big)\,\leq\, C_{\eta}\, \sigma(Q)\,,
\quad \forall Q\in \dd(E)\,.$$
\item For each ${\sbf^*}$, there is a regular Lip(1/2,1) graph $\Gamma_{{\sbf^*}}$, where the function defining the graph has Lip(1/2,1) constant
at most $\eta$ (that is, $b_1 \le \eta$) and whose half-order time derivative has P-BMO norm bounded by $b_2$, such that, for every $Q\in {\sbf^*}$,
\begin{equation}\label{eq2.2a}
\sup_{(t,X) \in KQ} \dist((t,X),\Gamma_{{\sbf^*}} )\,
+\,\sup_{(s,Y) \in B_Q^*\cap\Gamma_{{\sbf^*}}}\dist((s,Y),E) < \eta\,\diam(Q)\,,
\end{equation}
where $B_Q^*:= B(x_Q,K\diam(Q))$.
\end{enumerate}
\end{lemma}

\begin{remark}\label{subregimessame.rmk}
Notice that if $\sbf$ is any coherent subregime\footnote{This means $\sbf \subseteq \sbf^*$ and $\sbf$ satisfies the coherency conditions in Definition \ref{d3.11}.} of $\sbf^*$ then item (3) holds for every $Q \in \sbf$. Also, note that below we may insist that
$K$ is large, but this should be interpreted as taking $\eta$ small.
\end{remark}

\begin{definition}[Whitney cubes and Whitney regions]\label{whitney.def}
Given an ADR set $E \subset \ree$ we let $\W(E^c)$ be a the standard (parabolic) Whitney decomposition of $E^c$, that is, $\W(E^c) = \{I_i\}$ is a collection of closed parabolic dyadic cubes with disjoint interiors, $\cup_{\W(E^c)} I_i = E^c$ and for each $I \in \W(E^c)$
\[4 \diam(I) \le \dist(4I, E) \le \dist(I,E) \le 100 \diam(I).\]
(A similar construction can be found in Lemma \ref{regL121.lem} below). For $\eta \ll 1 \ll K$ and $Q \in \mathbb{D}(E)$, we define
\[\mathcal{W}_Q(\eta,K)=\{I \in \W(E^c): \eta^{1/4} \diam(Q) \le \dist(I,E) \le \dist(I, Q) \le K^{1/4}\diam(Q)\}\]
and
\[\mathcal{W}^*_Q(\eta,K)=\{I \in \W(E^c): \eta^{4} \diam(Q) \le \dist(I,E) \le \dist(I, Q) \le K\diam(Q)\}.\]
Comparing volumes, we see that $\#\W_Q \le C(n, \eta, K)$ (here and in the sequel we use the notation
$\#A$ to denote the cardinality of a finite set $A$).
For $\eta \ll 1 \ll K$ and $Q \in \mathbb{D}(E)$,
we set
\[U_Q(\eta, K) = \bigcup_{I \in \mathcal{W}_Q(\eta,K)} I\]
and
\[U^*_Q(\eta, K) = \bigcup_{I \in \mathcal{W}^*_Q(\eta,K)} I.\]
\end{definition}

\begin{remark}
 The reader may readily verify that the Whitney regions $U_Q$ and $U_Q^*$ have bounded overlaps, that is,
\[\sum_{Q \in \dd(E)} 1_{U_Q}(t,X) + \sum_{Q \in \dd(E)} 1_{U^*_Q}(t,X)  \lesssim 1, \quad \forall (t,X) \in \ree,\]
where the implicit constant depend on dimension, ADR, $\eta$ and $K$.
\end{remark}

\section{Domain approximation in stopping time regimes}\label{domainapprox.sect}

In this section we assume that $E$ has a bilateral corona decomposition and we fix $\sbf$,
a coherent subregime of a stopping time regime $\sbf^*$ in the bilateral corona  decomposition (by Remark \ref{subregimessame.rmk} the same estimates hold for $\sbf$). Our goal is to construct a family of graphs  that approximate the set $E$ well in
the  sense of Lemma \ref{bilateralcorona.lem} (3)  but have the additional property that they lie `above' (or on) the set $E$ at the scale and location of the maximal cube $Q_{\sbf}$. Other important properties of the construction will also be established including containment properties with respect to the Whitney regions defined above (see Definition \ref{whitney.def}). In the sequel will often insist on further smallness of $\eta$ depending on dimension and the ADR constant for $E$.
Compared to \cite{HMM1}, the constructions outlined in this section are the main novelties of this paper.

Let $Q_\sbf: = Q(\sbf)$ be the maximal cube in the coherent subregime under consideration. Recall that $\sbf \subseteq \sbf^*$ and that there exists a regular Lip(1/2,1) graph, $\Gamma_{\sbf^*}$, such that Lemma \ref{bilateralcorona.lem} (3) holds for $\sbf^*$ and hence also for $\sbf$. Without loss of generality we may assume that the $t$-independent plane over which
 $\Gamma: = \Gamma_{\sbf^*}$ is defined, is $\rn \times \{0\}$. Let $f: \rn \to \re$ be the
regular Lip(1/2,1) function that defines $\Gamma_{\sbf^*}$, that is,
\[\Gamma := \Gamma_{\sbf^*}= \{(t,x',f(t,x')): (t,x') \in \rn\}. \]
We define the $\ree$-valued function
\[F(t,x') = (t,x',f(t,x')).\]

Inspired by constructions in \cite{DS1},
we define the `stopping time distance'
$d:\ree \to \re$
by\footnote{Note that we take the stopping time distance in the {\it subregime}.}
\[d[(t,X)]= \inf_{Q \in \sbf}[\dist((t,X), Q) + \diam(Q)].\]
Given $\alpha \in [7/8,31/32]$ we introduce
\[g_\alpha(t,x'):= f(t,x') + \eta^\alpha d[F(t,x')]\]
and
\[G_\alpha(t,x'): = (t,x', g_\alpha(t,x')).\]
As $\alpha \in [7/8,31/32]$, below we will drop the subscript $\alpha$ and all constants will be
independent of $\alpha$.

We first prove that $g$ is Lip(1/2,1).
\begin{lemma}\label{dgisbasicallydf.lem}
If $\eta^{7/8} \le 1/2$, then $g$ is a Lip(1/2,1) function with constant less than $3\eta^{\alpha}$, and the function
\[G(t,x') := (t,x',g(t,x'))\]
satisfies
\[(1/2) d[F(t,x')] \le d[G(t,x')] \le 2 d[F(t,x')].\]
\end{lemma}
\begin{proof}
Note first that $d$ is Lip(1/2,1) (on $\ree$) with constant no more than $1$, that is, $|d[(t,X)] - d[(s,Y)]| \le \dist((t,X),(s,Y))$.
This follows from the fact that $d$ is the infimum of non-negative Lip(1/2,1) functions with constant $1$. Using this we see that
\begin{align*}
    |g(t,x') - g(s,y')| &\le |f(t,x') - f(s,y')|
    + \eta^\alpha |d[(t,x',f(t,x'))] - d[(s,y',f(s,y')]|
    \\ & \le \eta[|t-s|^{1/2} + |x' - y'|]
    \\ & \quad + \eta^{\alpha}[|t-s|^{1/2} + |x' - y'| + |f(t,x') - f(s,y')|]
    \\ & \le 3\eta^\alpha[|t-s|^{1/2} + |x' - y'|].
\end{align*}
To deduce the inequalities involving $d[G(t,x')]$ and $d[F(t,x')]$ we consider two cases. If $d[F(t,x')] = 0$ then $G(t,x') = F(t,x')$ so that $d[G(t,x')] = 0$. Otherwise, $d[F(t,x')] > 0$,
and using that $d$ is Lip(1/2,1) with constant $1$, we have
\begin{align*}|d[F(t,x')] - d[G(t,x')]| &\le \dist(F(t,x'), G(t,x')) = |f(t,x') - g(t,x')|
\\ & \le \eta^{\alpha} d[F(t,x')] \le (1/2) d[F(t,x')].
\end{align*}
From this we easily obtain
\[(1/2) d[F(t,x')] \le d[G(t,x')] \le 2 d[F(t,x')]. \]
\end{proof}

We will use the following elementary lemma several times.
\begin{lemma}\label{disttoflatgraph.lem} If $\Gamma'$ is the graph of a Lip(1/2,1) function $\varphi$ with Lip(1/2,1) norm less than $1/2$, then
\[(1/2)|x_n - \varphi(t,x')| \le \dist\big((t,X), \Gamma'\big) \le |x_n - \varphi(t,x')|,\]
 for all $(t,X) = (t,x',x_n)$.
\end{lemma}
\begin{proof}
The inequality on the right hand side is trivial. To prove the inequality on the left hand side,  we can, after a translation, assume
that  $(t,x',\varphi(t,x')) = (0,0,0)$. Furthermore, we can without loss of generality assume that $x_n \ge 0$ (the case $x_n < 0$ is treated in the
same way). Then  $|x_n - \varphi(0,0)| = x_n$. If $(s,y') \in \rn$ satisfies $|y'| + |s|^{1/2} > x_n$, then
\[\dist((t,X), (s,y', \varphi(s,y')) \ge |y'| + |s|^{1/2} \ge x_n.\]
If $(s,y') \in \rn$ satisfies $|y'| + |s|^{1/2} \le x_n$, then $|\varphi(s,y')| \le (1/2) x_n$ and hence
\[\dist((t,X), (s,y',\varphi(s,y')) \ge |x_n - \varphi(s,y')| \ge (1 - 1/2) x_n = 1/2 x_n.\]
These estimates prove the lemma.
\end{proof}

We will need the following properties of the stopping time distance.

\begin{lemma}\label{stdistutil.lem} Let $A > 1$.  If  $(t,X) \in \ree$ satisfies $0 < 2d[(t,X)] \le A \diam(Q_\sbf)$, then there exists $Q^* \in \sbf$ such that
\begin{equation}\label{stdistutileq1.eq}
\dist((t,X), Q^*) \le 2d[(t,X)] \le A\diam(Q^*) \le C_{n,ADR} d[(t,X)].
\end{equation}
If $d[(t,X)] =0 $,  then there exists, for every $\epsilon \in (0, A\diam(Q_\sbf))$, $Q_\epsilon \in \sbf$ such that
\begin{equation}\label{stdistutileq2.eq}
\dist((t,X), Q_\epsilon) \le \epsilon < A\diam(Q_\epsilon) \le C_{n,ADR} \epsilon.
\end{equation}
\end{lemma}
\begin{proof}
We start with proving \eqref{stdistutileq1.eq}. By definition there exists $Q \in \sbf$ such that
\[\dist((t,X), Q) + \diam(Q) \le 2d[(t,X)].\]
Let $Q^* \in \sbf$ be the smallest cube satisfying $Q \subseteq Q^* \subseteq Q_\sbf$ such that
\begin{equation}\label{Acubed.eq}
    A\diam(Q^*) \ge 2d[(t,X)].
\end{equation}
Such a cube exists because $Q_\sbf$ is a `candidate'. Notice that since $Q^*$ contains $Q$, $\dist((t,X),Q^*) \le \dist((t,X),Q) \le 2d[(t,X)]$ which proves the first inequality in \eqref{stdistutileq1.eq}.
The second inequality in \eqref{stdistutileq1.eq} holds by the choice of $Q^*$. To see that the last inequality holds, we first note that if $Q^* = Q$, then $\diam(Q^*) = \diam(Q) \le 2d[(t,X)]$ and we are done. Otherwise, the child of $Q^*$ containing $Q$, $Q'$, fails to satisfy \eqref{Acubed.eq} and hence
\[A \diam(Q^*) \lesssim_{n,ADR} A \diam(Q')  \le 2d[(t,X)]\,.\]
Since $A > 1$, it holds that $\diam(Q^*) \lesssim_{n,ADR} d[(t,X)]$ (with
the implicit constant independent of $A$). This proves \eqref{stdistutileq1.eq}.

To verify \eqref{stdistutileq2.eq}, note that by definition there exists $Q \in \sbf$ such that
\[\dist((t,X),Q) + \diam(Q) \le \epsilon \le A \diam(Q_\sbf).\]
This allows us to repeat the argument above to produce $Q_\epsilon$.
\end{proof}

\begin{lemma}\label{dfcloseeta.lem}
If $(t,X) \in B( (t_{Q_\sbf}, X_{Q_\sbf}), (1/4)K\diam(Q_{\sbf})) \cap E$ with $(t,X) = (t,x',x_n)$ then
\[\dist((t,X) , \Gamma) \lesssim \eta d[(t,X)]\]
and
\[|x_n - f(t,x')| \lesssim \eta d[(t,X)].\]
Here the implicit constants depend only on dimension and $ADR$.
\end{lemma}
\begin{proof}
The second inequality follows from the
first and Lemma \ref{disttoflatgraph.lem}. If $d[(t,X)] = 0$,
then Lemma \ref{stdistutil.lem} gives that for $n \in \mathbb{N}$, we have $(t,X) \in KQ_{1/n}$ for some
$Q_{1/n} \in \sbf$ with $\diam(Q_{1/n}) \approx 1/n$. Then using Lemma \ref{bilateralcorona.lem}(3) we have $\dist((t,X),\Gamma) \lesssim 1/n$ for all $n$ and hence $(t,X) \in \Gamma$. This proves the lemma in the case $d[(t,X)]  = 0$.

Now assume $d[(t,X)] > 0$ and note that that $2d[(t,X)] < (K-1)\diam(Q_\sbf)$, if $K > 6$. Applying Lemma \ref{stdistutil.lem} there exists $Q^*$ such that
\[\dist((t,X), Q^*) \le (K-1) \diam(Q^*) \lesssim d[(t,X)].\]
Then Lemma \ref{bilateralcorona.lem}(3) gives
\[\dist((t,X), \Gamma) \le \eta\diam(Q^*) \lesssim \eta d[(t,X)],\]
as desired.
\end{proof}

Let
\[\Gamma^+ := \{(t,x',g(t,x')): (t,x') \in \rn \}\]
denote the graph of $g$. We first prove that we did not lose too much by modifying $f$ and that, in fact, $E$ lies below $\Gamma^+$ (near $Q_\sbf$).

\begin{lemma}\label{dgcloseeta.lem}
If $(t,X) \in B((t_{Q_\sbf}, X_{Q_\sbf}), (1/4)K \diam(Q_\sbf)) \cap E$ with $(t,X) = (t,x',x_n)$, then
\begin{itemize}
    \item[(a)]
$\frac{1}{8}\,\eta^\alpha d[(t,X)]\, \le\, \dist((t,X), \Gamma^+) \,\le  \, 3\eta^{\alpha} d[(t,X)],$ and
    \item[(b)] $x_n \le g(t,x') \,-\, \frac{1}{4} \,\eta^\alpha d[(t,X)].$
\end{itemize}
\end{lemma}
\begin{proof}
If $d[(t,X)] = 0$, then $d[F(t,x')] = 0$  by Lemma \ref{dfcloseeta.lem}, and
\[(t,X) = (t,x',f(t,x')) = (t,x',{g(t,x')}).\]
This implies  (a) and (b).

Assume $d[(t,X)] > 0$. Lemma \ref{dfcloseeta.lem} yields the estimate
\begin{equation}\label{eq3.9}
|x_n - f(t,x')| \le C\eta d[(t,X)].
\end{equation}
If $C\eta < 1/2$,
then following the lines of
the proof of Lemma \ref{dgisbasicallydf.lem}, we have that
\begin{equation}\label{dfdz.eq}
(1/2) d[(t,X)] \le d[F(t,x')] \le 2d[(t,X)].
\end{equation}
Thus, by definition of $g$,
\begin{align*}
    g(t,x') - x_n &= \eta^{\alpha}d[F(t,x')] + (f(t,x') - x_n)
    \\& \ge \frac{\eta^{\alpha}}{2} d[(t,X)] - C\eta d[(t,X)] \ge \frac{\eta^\alpha}{4}d[(t,X)],
\end{align*}
provided $C \eta \le \eta^{\alpha}/4$.
This proves (b), and when combined with Lemma \ref{disttoflatgraph.lem}, it
gives the lower bound in (a). To verify the upper bound in (a), we use \eqref{eq3.9} and
\eqref{dfdz.eq} to write
\begin{multline*}
    |g(t,x') - x_n| \,\le\,  \eta^\alpha d[F(t,x')] + |f(t,x') - x_n| \\[4pt]
    \le \, 2\eta^\alpha d[(t,X)] \,+ \,C\eta d[(t,X)] \le 3\eta^\alpha d[(t,X)].
\end{multline*}
\end{proof}

We remind the reader that we have previously
defined certain Whitney regions (see Definition \ref{whitney.def}). We now investigate how these Whitney regions interact with the graphs we are constructing. First we need to see how they interact with the original graph $\Gamma$.
As in the elliptic setting \cite{HMM1}, we have the following.

\begin{lemma}\label{WhitneyaboveGamma.lem}
If $Q \in \sbf$ and $I \in \W_Q$ then $I$ is either above or below $\Gamma$ (it does not meet $\Gamma$). Moreover, the we have the estimate
\[\dist(I, \Gamma) \ge \eta^{1/2} \diam(Q).\]
\end{lemma}
\begin{proof}
The first statement, about the cubes being above or below the graph, follows from the estimate. Suppose for the sake of contradiction that there exists $I \in \W_Q$, $Q \in \sbf$ such that $\dist(I,\Gamma) < \eta^{1/2}\diam(Q)$ and let $(s,Y) \in \Gamma$ be such that $\dist((s,Y),I) \le \eta^{1/2}\diam(Q)$. By construction
$\dist((t,Z),(t_Q,X_Q)) \lesssim K^{1/4} \diam(Q)$ for all $(t,Z) \in I$ and hence
\[
\dist((s,Y), (t_Q, X_Q)) \le \eta^{1/2}\diam(Q) + CK^{1/4}\diam(Q) \lesssim K^{1/4}\diam(Q)\,.
\]
Then by Lemma \ref{bilateralcorona.lem}(3),
$\dist((s,Y),E) \le \eta\diam(Q)$. Choosing $(t_0,Z_0) \in I$
such that $\dist((t_0,Z_0),(s,Y)) =\dist((s,Y),I) \le \eta^{1/2} \diam(Q)$, we have that
\begin{multline*}
\dist(I,E) \,\le\, \dist((t_0,Z_0),(s,Y)) \,+\, \dist((s,Y),E)  \\[4pt]
 \le\, \eta^{1/2}\diam(Q) \,+\, \eta \diam(Q) \le 2\eta^{1/2} \,< \,\eta^{1/4},
\end{multline*}
provided $\eta^{1/4}< 1/2$. This violates that $I \in \W_Q$.
\end{proof}

In light of Lemma \ref{WhitneyaboveGamma.lem}, for $Q \in \sbf$ we have that $\W_Q = \W_Q^+ \cup \W_Q^-$ where $\W_Q^+$ is the collection of Whitney cubes above $\Gamma$ and $\W_Q^-$ is the collection of Whitney cubes below $\Gamma$. We then define
\[U_Q^\pm := \bigcup_{I \in \W_Q^\pm} I.\]

The following lemma says that the $U_Q^+$ still lies above $\Gamma^+$ and, when $(t,X) \in U_Q^+$ the distance from $(t,X)$ to $\Gamma^+$ is roughly the distance to $E$.

\begin{lemma}\label{whitneygammaplus.lem}
Let $Q \in \sbf$. If $\eta$ is sufficiently small, then $U_Q^+$ lies above $\Gamma^+$ and
\begin{equation} \label{eq3.13}
x_n - g(t,x') \,\ge\,
(1/2)\dist\big((t,X),\Gamma\big), \quad \forall (t,X) = (t,x',x_n) \in U_Q^+.
\end{equation}
Moreover,
\begin{equation} \label{eq3.14}
\dist\big((t,X),\Gamma^+\big) \approx \dist \big((t,X),E\big) \quad (t,X) \in U_Q^+,
\end{equation}
where the implicit constants depend on dimension, ADR, $\eta$ and $K$.
\end{lemma}
\begin{proof}
Recall that $\eta = K^{-1}$. Let $(t,X) = (t,x',x_n) \in I$ for some $I \in \W_Q^+$.  As $\dist((t,X),(t_Q,X_Q)) \lesssim K^{1/4} \diam(Q)$ and $Q \in \sbf$ we have
\begin{equation} \label{eq3.15}
\dist((t,X), \Gamma) \lesssim (K^{1/4} + \eta)\diam(Q) \lesssim K^{1/4}\diam(Q).
\end{equation}
Using Lemma \ref{disttoflatgraph.lem}, $$|(t,X) - F(t,x')| \lesssim K^{1/4} \diam(Q),$$
 and therefore $\dist(F(t,x'), Q) \lesssim K^{1/4}\diam(Q)$. It follows that $d[F(t,x')] \lesssim K^{1/4} \diam(Q)$, and using Lemma \ref{WhitneyaboveGamma.lem}
\begin{multline*}
    x_n - f(t,x') \ge \dist((t,X),\Gamma) \,\ge\, \eta^{1/2}\diam(Q) \\[4pt]
    \gtrsim \,\eta^{1/2}K^{-1/4} d[F(t,x')]\, \approx \,\eta^{3/4} d[F(t,x')].
\end{multline*}
By definition of $g$ and that $\alpha \ge 7/8$, Lemma \ref{disttoflatgraph.lem} implies that
\begin{multline}\label{eq3.16}
x_n - g(t,x') = x_n - f(t,x') - \eta^\alpha d[F(t,x')] \ge (1/2)(x_n - f(t,x'))
\\[4pt]
\ge \, (1/2) \dist((t,X),\Gamma) \,\ge\, (1/2) \eta^{1/2}\diam(Q),
\end{multline}
where the next-to-last inequality yields \eqref{eq3.13}, and where we
have used Lemma \ref{WhitneyaboveGamma.lem} in the last step.
In particular, $U_Q^+$ lies above $\Gamma^+$.  Using
\eqref{eq3.15} and the last inequality in \eqref{eq3.16},
and then the properties of the Whitney cubes in $\W_Q$, we have
\[  \dist((t,X), \Gamma) \approx_{\eta,K} \diam(Q)  \approx_{\eta,K} \dist((t,X),E) \,.\]
Combining \eqref{eq3.16} and the last displayed estimate,
and using Lemma \ref{disttoflatgraph.lem} we obtain
\[\dist((t,X),\Gamma^+) \ge (1/2)(x_n - g(t,x')) \ge (1/4) \dist((t,X),\Gamma) \approx_{\eta,K} \dist((t,X),E) \]
and
\[\dist((t,X),\Gamma^+) \le x_n- g(t,x') \le x_n - f(t,x') \le 2 \dist((t,X),\Gamma) \approx \dist((t,X),E).\]
This proves the lemma.
\end{proof}

We  also require that close to $Q_\sbf$, the region above $\Gamma^+$ shall be
contained in a collection of Whitney regions associated to $Q \in \sbf$. This can be done using the Whitney regions $U_Q^*$.

\begin{lemma}\label{domainsinbigcubes.lem}
Suppose $(t,X) = (t,x',x_n)$ satisfies $x_n > g(t,x')$ and
\begin{equation*}
(t,X) \in B\big((t_{Q_\sbf},X_{Q_\sbf}), (1/32)K\diam(Q_{\sbf})\big)\,.
\end{equation*}
Then
\begin{equation}\label{domainsinbigcubeseq1.eq}
\dist\big((t,X),E\big) \ge \dist\big((t,X), \Gamma^+\big)
\end{equation}
and
there exists $Q^* \in \sbf$ such that $(t,X) \in U^*_{Q^*}$.
\end{lemma}
\begin{proof}
Let $(t,X)$ be as above. By Lemma \ref{dgcloseeta.lem}, we see that $d[(t,X)] \neq 0$.
To prove \eqref{domainsinbigcubeseq1.eq}, we note that if $(s,Y)$ is the
closest point to $(t,X)$ in $E$, then
\[
\dist\big((t,X),(s,Y)\big) \leq \dist((t,X),(t_{Q_\sbf},X_{Q_\sbf})) < (K/32) \diam(Q_{\sbf})\,.
\]
Thus $\dist((s,Y),(t_{Q_\sbf},X_{Q_\sbf})) < (K/16)\diam(Q_{\sbf})$, so in particular,
\[
(s,Y) \in B\big((t_{Q_\sbf},X_{Q_\sbf}), (K/4) \diam(Q_{\sbf})\big)\cap E\,.
\]
By Lemma \ref{dgcloseeta.lem}(b), $(s,Y)$ lies below $\Gamma^+$ and hence the line segment between $(s,Y)$ and $(t,X)$ meets $\Gamma^+$. This proves \eqref{domainsinbigcubeseq1.eq}.

To prove the existence of  $Q^* \in \sbf$ such that $(t,X) \in U^*_{Q^*}$ we break the proof into cases.
\\
{\bf Case 1}: $x_n - g(t,x') \ge \eta^3d[(t,X)]$.
In this case, by \eqref{domainsinbigcubeseq1.eq} and Lemma \ref{disttoflatgraph.lem},
\[
\dist\big((t,X),E\big)\, \ge \,  \dist\big((t,X),\Gamma^+\big)\,\ge\, (1/2)\,\eta^3 d[(t,X)]\,.
\]
Since $d[(t,X)] \le (K/32 + 1)\diam(Q_{\sbf}) < ((K-1)/2)\diam(Q_\sbf)$,
we may use Lemma \ref{stdistutil.lem} to produce $Q^*$ with
\[\dist\big((t,X),Q^*\big) \le (K-1)\diam(Q^*) \lesssim_{n,ADR} d[(t,X)].\]
Thus,
\[(t,X) \in B\big((t_{Q^*},X_{Q^*}), K\diam(Q^*)\big)\]
and
\[\dist\big((t,X),E\big) \gtrsim \eta^3d[(t,X)] \gtrsim \eta^3 (K-1)\diam(Q^*) \approx \eta^{2} \diam(Q),\]
where the implicit constants depend on dimension and ADR. Letting $I \in \W$ be such that $(t,X) \in I$
it follows that $I \in \W^*_Q$, provided $\eta$ is sufficiently small.

\noindent
{\bf Case 2}: $x_n - g(t,x') \le \eta^3d[(t,X)]$.
In this case, note that
\begin{equation}\label{eq3.19}
\dist\big((t,X),G(t,x')\big) \leq x_n - g(t,x') \le \eta^3d[(t,X)] \,.
\end{equation}
Thus, since $d$ is Lipschitz with norm 1 with respect to $\dist(\cdot)$, we have,  for $\eta^3 < 1/2$,
\begin{equation}\label{eq3.20}
(1/2) d[(t,X)] \le d[G(t,x')] \le 2d[(t,X)].
\end{equation}
In particular, $d[G(t,x')] > 0$. Notice then that
\[
d[G(t,x')] \le 2d[(t,X)]  \le (K/16 + 2)\diam(Q_{\sbf}) \le \tfrac{K-1}{8}\diam(Q_{\sbf})\,,
\]
provided that $K$ is large enough, and Lemma \ref{stdistutil.lem} then
yields $Q^* \in \sbf$ such that
\begin{equation}\label{eq3.21}
\dist(G(t,x'), Q^*) \le \frac{(K-1)}{4} \diam(Q^*) \approx d[G(t,x')].
\end{equation}
Combining the latter estimate with \eqref{eq3.19} and \eqref{eq3.20}, we see that
\[(t,X) \in B\big((t_{Q^*},X_{Q^*}), ((K-1)/2)\diam(Q^*)\big).\]
\begin{claim}\label{gorderdg.eq} For $\eta$ chosen small enough ($\eta^2 < 1/2$ will suffice at this stage),
\[\dist(G(t,x'), E) \ge \eta^2 d[G(t,x')]\,.\]
\end{claim}
Taking the claim for granted momentarily,  by
\eqref{eq3.19}, \eqref{eq3.20}, and \eqref{eq3.21}, we have \begin{align*}
\dist((t,X), E) & \ge  \dist(G(t,x'),E)  - \dist((t,X), G(t,x'))
\\&\ge \dist(G(t,x'),E) - \eta^3d[(t,X)] \gtrsim \eta^2 d[(t,X)] \approx \eta^2[G(t,x')]
\\ &\gtrsim \eta^2(K-1) \diam(Q^*) \approx \eta  \diam(Q^*)
\end{align*}
and the lemma is proved. It remains to prove Claim \ref{gorderdg.eq}.
\begin{proof}[Proof of Claim \ref{gorderdg.eq}]
Let $(s,Y)= (s,y',y_n) \in E$ be such that
\[
\dist( G(t,x'),(s,Y)) = \dist(G(t,x'), E)\,.
\]
Assume, for the sake of obtaining a contradiction, that 
\[\dist(G(t,x'),(s,Y)) < \eta^2 d[G(t,x')].\]
Then for
$\eta^2 < 1/2$, since $d$ is Lipschitz with norm 1 with respect to $\dist(\cdot)$, we have
\[d[G(t,y')] \le 2 d[(s,Y)]\,.\]
Hence, under the current assumption that $\dist(G(t,x'),(s,Y)) < \eta^2 d[G(t,x')]$,
\[|y' - x'| + |t-s|^{1/2} < 2\eta^{2}d[(s,Y)].\]
Since $g$ is Lip(1/2,1) with constant $3\eta^{\alpha}$ (in particular less than 1)
\[\dist(G(t,x'), G(s,y'))\lesssim \eta^2 d[(s,Y)].\]
Thus,
\[
|y_n - g(s,y')| = \dist\big((s,Y) , G(s,y')\big) \lesssim \eta^{2}d[(s,Y)],\]
which contradicts the conclusion of Lemma \ref{dgcloseeta.lem}, provided that
\begin{equation} \label{eq3.23}
(s,Y) \in B\big((t_{Q_\sbf},X_{Q_\sbf}), (K/4)\diam(Q_{\sbf})\big)\,.
\end{equation}
Indeed, the latter is true, as we now show.
Recall that by hypothesis
\begin{equation*}
(t,X) \in B\big((t_{Q_\sbf},X_{Q_\sbf}), (K/32)\diam(Q_{\sbf})\big)\,.
\end{equation*}
Moreover, in the scenario of Case 2,
\[
\dist\big((t,X) ,G(t,x')\big)= x_n - g(t,x') < \eta^3 d[(t,X)] \le \diam(Q_\sbf)\,,
\]
and therefore \[\dist\big(G(t,x'), (t_{Q_\sbf}, X_{Q_\sbf})\big) < (K/16)\diam(Q_\sbf).\]
Since $(s,Y)$ is the closest point on $E$ to $G(t,x')$, it must be that
\[(s,Y) \in B\big((t_{Q_\sbf},X_{Q_\sbf}), (K/8)\diam(Q_\sbf)\big)\,.\]
In particular, \eqref{eq3.23} holds, and this proves the claim.
\end{proof}
\end{proof}

Our next goal is to produce a regular version of the graphs we have constructed above. The vehicle for this regularization is the following lemma.

\begin{lemma}\label{regL121.lem}
Let $h: \rn \to \re$, $h(t,x') \ge 0$, be a Lip(1/2,1) function with Lip(1/2,1) constant (at most) $1$. There exists a function $H: \rn \to \re$ such that
\begin{enumerate}
    \item $c_1h(t,x') \le H(t,x') \le c_2 h(t,x')$ for all $(t,x') \in \rn$.
    \item If $Z = \{(t,x'): h(t,x') = 0\}$ then
    \[h(t,x')^{2m-1}|\partial_t^m H(t,x')| + h(t,x')^{m-1}|\nabla_{x'}^m H(t,x')| \le c_{n,m}, \quad \forall (t,x') \in Z^c, m \in \mathbb{N}.\]
    \item $H \in Lip(1/2,1)$ with constant less than $c_3$.
\end{enumerate}
Here $c_1,c_2,c_3$ depend on dimension alone and $c_{n,m}$ depends on dimension and $m$.
Moreover, $H$ enjoys the estimate
\[\|D_t^{1/2}H\|_{P\text{-}BMO} \le c_4,\]
where $c_4$ depends only on dimension.
\end{lemma}
The proof has many `standard' elements (if one knows where to look), but is a little lengthy.  The proof can be found in the appendix.

Now we are ready to create our regularized graph.
Let $h(t,x') := (1/2)d[F(t,x')]$ and let $H(t,x')$ be the function provided\footnote{See the proof of Lemma \ref{dgisbasicallydf.lem}, from which one can easily deduce that $d[F(t,x')]$ has Lip(1/2,1) norm less than $1 + \eta$ and hence $h(t,x')$ has Lip(1/2,1) norm less than $1$. This allows one to apply Lemma \ref{regL121.lem}.} by Lemma \ref{regL121.lem}. We define two functions
\[\psi^\pm_{\eta,\sbf} (t,x) := f(t,x) \pm \eta^{15/16}H(t,x').\]
We hope that it is clear to the reader that the function
\[g^-_\alpha (t,x') := f(t,x') - \eta^\alpha d[F(t,x')]\]
has properties analogous to those of $g_\alpha(t,x')$ except that $g_\alpha^-$ is below $f$ and $E$, that the cubes in $\W_Q^-$ are below the graph of $g_\alpha^-$, etc..  We next deduce that $\psi^\pm_{\eta, \sbf}$ has the same properties as the functions $g_\alpha$ (and $g_\alpha^-$), enumerated  below.

\begin{proposition}\label{stopdomainconst.prop}
Let $E$ be uniformly rectifiable in the parabolic sense. Let $\dd(E) = \G \cup \B$, $\G = \cup_{\sbf^* \in \mathcal{S}} \sbf^*$, $\{\Gamma_{\sbf^*}\}_{\sbf^* \in \mathcal{S}}$ be the  bilateral corona  decomposition of $E$ given by Lemma \ref{bilateralcorona.lem}, with constants $\eta \ll 1$ and $K = \eta^{-1}$, and $b_2$. Let $M_0$ be the constant from
Lemma \ref{pullbackCME.lem} below, with $\tilde{b}_1 = 2$ and $\tilde{b}_2 = 1 + b_2$. If $\eta$ is sufficiently small, depending only on dimension and ADR, then the following holds.

Let $\sbf^* \in \mathcal{S}$. Then for every coherent subregime $\sbf$ of $\sbf^*$,
there is a $t$-independent plane $P_\sbf$ and two regular parabolic graphs  $\Gamma_\sbf^\pm$ over $P_\sbf$ given by functions $\psi^\pm_{\eta,\sbf}$, with $\|\psi^\pm_{\eta,\sbf}\|_{Lip(1/2,1)} \le C_{n} \eta^{15/16}$ and $\|D_t^{1/2}\psi^\pm_{\eta,\sbf}\|_{P\text{-}BMO} \le (1 + b_2)$, with the following properties (in the coordinates given by $P_\sbf \oplus P_\sbf^\perp$):

\begin{enumerate}

   \item If $Q \in \sbf$, then $\W_Q$ has a disjoint decomposition $\W_Q = \W_Q^+ \cup \W_Q^-$, and if we let $U_Q^\pm := \cup_{I \in \W_Q^\pm} I$, then
  \[U_Q^\pm \subseteq B((t_{Q_\sbf}, X_{Q_\sbf}), K^{3/4}\diam(Q_\sbf)) \cap \{\pm x_n > \pm\psi^\pm_{\eta,\sbf}(t,x')\}.\]
    Here the notation $\{\pm x_n > \pm \psi^\pm_{\eta,\sbf}(t,x')\}$ means
    \[\{(t,X)= (t,x',x_n): \pm x_n > \pm\psi^\pm_{\eta,\sbf}(t,x')\}.\]
    In particular,
    \[\bigcup_{Q \in \sbf} U_Q^\pm \subseteq B((t_{Q_\sbf}, X_{Q_\sbf}), K^{3/4}\diam(Q_\sbf)) \cap \{\pm x_n > \pm\psi^\pm_{\eta,\sbf}(t,x')\}.\]
    \item If $(t,X) \in \bigcup_{Q \in \sbf} U_Q^\pm$ then
    \[\dist((t,X),E) \approx_\eta \dist((t,X),\Gamma_\sbf^\pm).\]
    \item If $  (t,X) \in B((t_{Q_\sbf}, X_{Q_\sbf}), (K/32)\diam(Q_\sbf)) \cap \{\pm x_n > \pm\psi^\pm_{\eta,\sbf}(t,x')\} $, then
     \[
\dist((t,X),E) \ge \dist((t,X),\Gamma_\sbf^\pm).
\]
    \item
  $B((t_{Q_\sbf}, X_{Q_\sbf}), (K/32)\diam(Q_\sbf)) \cap \{\pm x_n > \pm\psi^\pm_{\eta,\sbf}(t,x')\} \subseteq \bigcup_{Q \in \sbf} U_Q^*. $

  \medskip

    \item There exist  $(t_\sbf^\pm, X_\sbf^\pm) \in \Gamma_\sbf^\pm$ such that
        \begin{equation}\label{prop5a.eq}
    B((t_{Q_\sbf}, X_{Q_\sbf}), M_0 K^{3/4}\diam(Q_\sbf)) \subset
    B((t_\sbf^\pm, X_\sbf^\pm), K^{7/8}\diam(Q_\sbf))
    \end{equation}
    and
    \begin{equation}\label{prop5b.eq}
        B((t_\sbf^\pm, X_\sbf^\pm), K^{7/8}\diam(Q_\sbf))\subset B((t_{Q_\sbf}, X_{Q_\sbf}), (K/32)\diam(Q_\sbf)).
    \end{equation}
\end{enumerate}
\end{proposition}
\begin{proof}
Let $\psi^\pm_{\eta,\sbf}$ be as constructed before the statement of the proposition. Both $\psi^\pm_{\eta,\sbf}$ are  Lip(1/2,1) with constant less than $C\eta^{15/16}$ because $f$ is Lip(1/2,1) with constant less than $\eta$ and $H$ has Lip(1/2,1)
norm less than $c_3= c_3(n)$ (see Lemma \ref{regL121.lem}). Similarly,
\begin{multline}
\|D_t^{1/2}\psi^\pm_{\eta,\sbf}\|_{P\text{-}BMO} \le \|D_t^{1/2}f\|_{P\text{-}BMO} + \eta^{15/16} \|D_t^{1/2}H\|_{P\text{-}BMO}
\\ \le \|D_t^{1/2}f\|_{P\text{-}BMO} + c_4\eta^{15/16} <  b_2 +1,
\end{multline}
provided $\eta$ is sufficiently small.
We define $\Gamma_\sbf^\pm$ to be the graphs of $\psi^\pm_{\eta,\sbf}$ (resp.).

We claim that
 \begin{equation}\label{gsmoothedcomp.eq}
g_{7/8}(t,x') \ge  \psi^+_{\eta,\sbf}(t,x') \ge g_{31/32}(t,x')\,,
\end{equation}
and
 \begin{equation}\label{gsmoothedcompminus.eq}
g^-_{7/8}(t,x') \le  \psi^-_{\eta,\sbf}(t,x') \le g^-_{31/32}(t,x').
\end{equation}
Indeed, these inequalities are a result of the fact that
 $\eta^{7/8} \gg \eta^{15/16} \gg \eta^{31/32}$
when $\eta$ is very small,
 and the properties of $H$ in relation to $d$. For example,
\[ \psi^+_{\eta,\sbf}(t,x') - g_{31/32}(t,x') = \eta^{15/16} H(t,x') - \eta^{31/32}d[F(t,x')].\]
Using Lemma \ref{regL121.lem} we have
\[d[F(t,x')] = 2h(t,x') \approx_n H(t,x').\]
Since the constants are independent of $\eta$ the
second inequality in \eqref{gsmoothedcomp.eq} follows. The other inequalities are treated similarly.

With \eqref{gsmoothedcomp.eq}-\eqref{gsmoothedcompminus.eq} at hand, properties (1) and (2) can be deduced directly from Lemmas \ref{WhitneyaboveGamma.lem} and \ref{whitneygammaplus.lem}. Note that to
prove property (2), we observe that $(t,X) \in U_Q$ implies that $(t,X)$ is above the graphs of {\it both} $g_{7/8}$ and $g_{31/32}$. Similarly, properties (3) and (4) can be deduced from \eqref{gsmoothedcomp.eq}
(or \eqref{gsmoothedcompminus.eq}), and Lemma \ref{domainsinbigcubes.lem}:
to prove (3) and (4) in, e.g., the context of $\psi^+_{\eta,\sbf}(t,x')$,
we simply observe that if $(t,X)$ is above $\psi^+_{\eta,\sbf}(t,x')$, then it is above $g_{31/32}$.

To prove (5), we let $(s,Y)$ be the closest point on $\Gamma_{\sbf}$ to $(t_{Q_\sbf},X_{Q_\sbf})$ and we observe from
Lemma \ref{bilateralcorona.lem}(3) that
\[
\dist\big((t_{Q_\sbf},X_{Q_\sbf}),(s,Y)\big) \le \eta \diam(Q_\sbf)\,.
\]
As $(s,Y) = (s,y',y_n) = F(s,y')$ we have that
\[H(s,y') \approx_{n} d[F(s,y')] \le \diam(Q_\sbf) + \eta \diam(Q_\sbf) < 2 \diam(Q_\sbf),\]
where we have used the properties of $H$ given by Lemma \ref{regL121.lem}. Then by definition
\begin{align*}
\dist\big((t_{Q_\sbf},X_{Q_\sbf}), (s,y',\psi^\pm_{\eta,\sbf}(s,y')\big)
&\le \dist\big((t_{Q_\sbf},X_{Q_\sbf}),(s,Y)\big) + \eta^{15/16}H(s,y')
\\& \lesssim_{n}\eta^{15/16}\diam(Q_\sbf).
\end{align*}
Setting $(t_\sbf^\pm, X_\sbf^\pm)  = (s,y',\psi^\pm_{\eta,\sbf}(s,y'))$, and taking $\eta$ sufficiently small (and hence $K$ sufficiently large),  yield (5).
\end{proof}

\section{Carleson Measure Estimates: Proof of Theorems \ref{PURCMEMT.thrm} and \ref{PURCMEMT2.thrm}}\label{proofmt.sect}

Before we get into the details of proving Theorems  \ref{PURCMEMT.thrm} and \ref{PURCMEMT2.thrm}, we point out that the domains we produced in Proposition \ref{stopdomainconst.prop} support (a local version of) the Carleson measure estimate.

\begin{lemma}[{\cite[Lemma A.2]{HL-JFA} }]\label{pullbackCME.lem}
 Let $\tilde{b}_1,\tilde{b}_2$ be fixed non-negative constants.
Suppose that $\varphi(t,x')$ is a regular Lip(1/2,1) function with Lip(1/2,1) constant $\tilde{b}_1$ and such that $\| D_t^{1/2} \varphi \|_{P\text{-}BMO} \le \tilde{b}_2$. Let
\[\om^+ = \{(t,X) = (t,x',x_n): x_n > \varphi(t,x')\}.\]
Then there exist $M_0 = M_0(n,\tilde{b}_1,\tilde{b}_2) > 1$ and $c_5 = c_5(n,\tilde{b}_1,\tilde{b}_2)$, such that if $u$ is a bounded solution to $(\partial_t - \Delta_x)u = 0$ in
\[\om^+((t_0,X_0),M_0r):= B((t_0,X_0), M_0r) \cap \om^+,\]
for some $(t_0,X_0)\in \pom$, then
\begin{equation}\label{localSleNtypeest.eq}
\iint\limits_{B((t_0,X_0),r) \cap \om^+} |\nabla u(s,Y)|^2 \tilde{\delta}(s,Y) \,  dY \, ds \le c_5 r^{n+1} \|u\|_{L^\infty(\om^+((t_0,X_0),M_0r))}^2.
\end{equation}
Here $\tilde{\delta}(s,Y) = \dist((s,Y), \pom^+)$.
An analogous statement holds for bounded solutions to
$(\partial_t - \Delta_x)u = 0$ in $\om^-((t_0,X_0),M_0r):= B((t_0,X_0), M_0r) \cap \om^-$, where
\[\om^- = \{(t,X) = (t,x',x_n): x_n < \varphi(t,x')\}.\]
\end{lemma}
\begin{proof} The lemma is a consequence of \cite[Lemma A.2]{HL-JFA}. However, to reduce the proof of the lemma to \cite[Lemma A.2]{HL-JFA} one has to note two things. First, by using the  parabolic version of the Dahlberg-Kenig-Stein pullback, the operator $(\partial_t - \Delta)$ is transformed
to an operator of the form treated in \cite[Lemma A.2]{HL-JFA} in the upper half space. Furthermore, $\om^+((t_0,X_0),M_0r)$ is transformed into a region containing a Carleson region of size roughly $M_0r$. Second, while stated for solutions in the upper half space,  \cite[Lemma A.2]{HL-JFA} only uses that $u$ is a solution in a Carleson region.
\end{proof}

Now we are ready to prove Theorem \ref{PURCMEMT.thrm}.
\begin{proof}[Proof of Theorem \ref{PURCMEMT.thrm}]
Let $E$ be  uniformly rectifiable in the parabolic sense and let $u$ be a
bounded solution to $(\partial_t - \Delta_X)u = 0$ in $E^c$. We may assume that $\|u\|_{L^\infty(E^c)} \neq 0$, since the conclusion of Theorem \ref{PURCMEMT.thrm} holds trivially if $\|u\|_{L^\infty(E^c)} = 0$. Let
\[v: = \frac{u}{\|u\|_{L^\infty(E^c)}}.\]
Then $\|v\|_{L^\infty(E^c)} = 1$ and it clearly suffices to prove the theorem with $v$ in place of $u$.

For each $Q \in \dd(E)$ we set
\[\beta_Q =\iint_{U_Q}|\nabla_X v|^2  \delta(s,Y) \, dY\, ds.\]
We first reduce the proof of the theorem to a statement concerning the $\beta_Q$'s.
\begin{claim}\label{betasgoodenough.cl}
If there exists $C$ (independent of $v$) such that
\begin{equation}\label{betasgooden.eq}
\sum_{Q \subseteq Q_0} \beta_Q \le C \sigma(Q_0), \quad \forall Q_0 \in  \dd(E),
\end{equation}
then there exists $C'$ such that
\begin{equation}\label{eq4.6}
\sup_{(t,X) \in E, r > 0} r^{-n-1}
\iint\limits_{B((t,X),r)} |\nabla_X v|^2  \delta(s,Y) \,  dY \, ds \le C' \,.
\end{equation}
In particular, to prove the theorem it is enough to verify \eqref{betasgooden.eq}
\end{claim}

\begin{proof}[Sketch of Proof of Claim \ref{betasgoodenough.cl}]
To prove that \eqref{betasgooden.eq} implies
 \eqref{eq4.6}, we select a collection $\{Q_0^i\}_i \subset \dd(E)$, such that, for each $i$, $\diam(Q_0^i) \approx \kappa r$, and such that the collection has uniformly
bounded cardinality depending only on $n$ and ADR. Furthermore, $B((t,X),r)\cap E\subset \cup_i Q_0^i$.
Choosing $\kappa$ large enough depending only on allowable parameters, we have that
$B((t,X),r)\setminus E \subset \cup_i \cup_{Q\subset Q_0^i} U_Q$.
We can then apply \eqref{betasgooden.eq} to each $Q_0=Q_0^i$.
We omit the routine details.
\end{proof}

We have now reduced everything to the setting of our dyadic machinery and we are almost ready to begin employing the constructions in Proposition \ref{stopdomainconst.prop}. Notice that these constructions are only likely to be helpful when bounding a $\beta_Q$ when $Q$ is a `good' cube. That is why the following claim is important when handling the `bad cubes'.
\begin{claim}\label{uniBqbound.eq}
There exists a constant $A$ depending only dimension, $K$, $\eta$ and ADR such that
\[\beta_Q \le A \sigma(Q).\]
\end{claim}

\begin{proof}[Sketch of Proof of Claim \ref{uniBqbound.eq}]
The claim follows readily from Lemma \ref{cacc} and ADR  as in \cite{HMM1}.
We omit the details.
\end{proof}

We now prove \eqref{betasgooden.eq}. Fix $Q_0 \in \dd(E)$.
If $Q_0 \in \sbf^*$ for some $\sbf^* \in \mathcal{S}$ we let $\sbf = \sbf^* \cap \dd_{Q_0}$ and note that 
$\sbf$ is a coherent subregime of $\sbf^*$ with maximal cube $Q_0$. $\dd_{Q_0}$ has the disjoint decomposition
\begin{equation}\label{DQdecomp.eq}\dd_{Q_0} = \{Q \in \B: Q \subseteq Q_0\} \cup
\big(\cup_{\sbf^*: Q(\sbf^*) \subset Q_0 } \sbf^*\big)
\cup \sbf,
\end{equation}
where  $\sbf = \emptyset$ if $Q_0$ is not in a stopping time regime (i.e., if $Q_0$ is a bad cube).
By Lemma \ref{bilateralcorona.lem}(2) and Claim \ref{uniBqbound.eq}
\begin{equation}\label{bdcubbd.eq}
\sum_{\substack{Q \subseteq Q_0 \\Q \in \B}} \beta_Q \le C\sum_{\substack{Q \subseteq Q_0 \\Q \in \B}} \sigma(Q) \le CC_{\eta,K} \sigma(Q_0).
\end{equation}
Let us suppose, for the moment, that we can show that
\begin{equation}\label{sbfpackbeta.eq}\sum_{Q \in \sbf^*} \beta_Q \le C\sigma\big(Q(\sbf^*)\big)\,,
\end{equation}
for all $\sbf^*$ such that $Q(\sbf^*) \subset Q_0$ and that
\begin{equation}\label{sbfprimepackbeta.eq}\sum_{Q \in \sbf} \beta_Q \le C\sigma\big(Q(\sbf)\big)
= C\sigma(Q_0), \end{equation}
if $Q_0$ is in some stopping time regime. Then, by Lemma \ref{bilateralcorona.lem}(2),
\begin{multline*}
\sum_{Q \in \sbf} \beta_Q \, + \sum_{\sbf^*: Q(\sbf^*) \subset Q_0} \sum_{Q \in \sbf^*} \beta_Q
\\[4pt]
 \lesssim \sigma(Q_0) \,+ \sum_{\sbf^*: Q(\sbf^*) \subseteq Q_0} \sigma\big(Q(\sbf^*)\big) \lesssim (C_{\eta,K} + 1) \sigma(Q_0).
\end{multline*}
Combining this estimate with \eqref{bdcubbd.eq}, and using the
decomposition of $\dd_{Q_0}$ in \eqref{DQdecomp.eq}, proves
\eqref{betasgooden.eq} and hence the theorem. Thus, it suffices to verify
\eqref{sbfpackbeta.eq} and \eqref{sbfprimepackbeta.eq}. In the following we only prove \eqref{sbfpackbeta.eq} as the
only change needed when proving \eqref{sbfprimepackbeta.eq} is to change $\sbf^*$ to $\sbf$.

To prove \eqref{sbfpackbeta.eq} we use Proposition \ref{stopdomainconst.prop}. Fix $\sbf^*$ such that $Q(\sbf^*) \subset Q$ and let $\psi_{\eta,\sbf^*}^\pm$ be the functions from Proposition \ref{stopdomainconst.prop}. Then by Proposition \ref{stopdomainconst.prop}(1), if $Q \in \sbf^*$ then $U_Q = U_Q^+ \cup U_Q^-$ and hence
$\beta_Q = \beta_Q^+ + \beta_Q^-,$
where
\[\beta_Q^\pm := \iint_{U_Q^\pm}|\nabla_X v|^2  \delta(s,Y) \, dY\, ds. \]
Clearly it is enough to show that
\[\sum_{Q \in \sbf^*} \beta_Q^\pm \le C\sigma(Q(\sbf^*)).\]
We prove the estimate for the sum of the $\beta_Q^+$ leaving the straightforward modification needed to handle the sum of the $\beta_Q^-$ to the interested reader.
Moreover, since the $\{U_Q\}$, and hence the $\{U_Q^\pm\}$, have bounded overlap it is enough to prove that
\begin{equation}\label{unUQplusCME.eq}
\iint\limits_{\bigcup_{Q \in \sbf^*} U_Q^+}|\nabla_X v|^2 \delta(s,Y) \,  dY\, ds  \le C\sigma(Q(\sbf^*)).
\end{equation}

Let
\[\om_{\sbf^*}^+ = B\big((t_{\sbf^*}^+,X_{\sbf^*}^+), K^{7/8}\diam(Q_{\sbf^*})\big)
\cap \{ x_n > \psi^+_{\eta,\sbf^*}(t,x')\} \,, \]
\[ \widetilde{\om}_{\sbf^*}^+ = B\big((t_{Q_{\sbf^*}},X_{Q_{\sbf^*}}), K^{3/4}\diam(Q_{\sbf^*})\big)
\cap \{ x_n > \psi^+_{\eta,\sbf^*}(t,x')\}\,,\]
and
\[ \widehat{\om}_{\sbf^*}^+ = B\big((t_{Q_{\sbf^*}},X_{Q_{\sbf^*}}), M_0 K^{3/4}\diam(Q_{\sbf^*})\big)
\cap \{ x_n > \psi^+_{\eta,\sbf^*}(t,x')\}\,,\]
where we remind the reader that we use the coordinates $P_{\sbf^*} \oplus P_{\sbf^*}^\perp$
and the notation $\{ x_n > \psi^+_{\eta,\sbf^*}(t,x')\}$ means
\[\{(t,x',x_n): x_n > \psi^+_{\eta,\sbf^*}(t,x')\}.\]
We note that $\widehat{\om}_{\sbf^*}^+\subset \om_{\sbf^*}^+$, by \eqref{prop5a.eq}.
Proposition \ref{stopdomainconst.prop}(4) and \eqref{prop5b.eq}
ensure that $\om_{\sbf^*}^+$ is an open subset of
    $E^c$ and hence $v$ is a solution in $\om_{\sbf^*}^+$.
Applying
Lemma \ref{pullbackCME.lem} we have that
\begin{equation}\label{omplusCME.eq}
    \iint\limits_{\widetilde{\om}_{\sbf^*}^+}|\nabla_X v|^2  \,
     \tilde{\delta}(s,Y)\, dY\, ds \lesssim \diam(Q(\sbf^*))^{n+1} \approx \sigma(Q(\sbf^*)),
    \end{equation}
where $\tilde{\delta}(s,Y) = \dist((s,Y),\Gamma^+_{\sbf^*})$ and $\Gamma^+_{\sbf^*}$ is the
graph of $\psi_{\sbf^*}^+$. Note that if $Q \in \sbf^*$, then by
Proposition \ref{stopdomainconst.prop}(1),
we have that $U_Q^+ \in \widetilde{\om}^+_{\sbf^*}$. Moreover by Proposition \ref{stopdomainconst.prop}(2),
it holds that $\tilde{\delta}(s,Y) \approx \delta(s,Y)$ in $\cup_{Q \in \sbf^*} U_Q^+$. Thus,
\begin{multline}
\iint\limits_{\bigcup_{Q \in \sbf^*} U_Q^+}|\nabla_X v|^2 \, \delta(s,Y) \,  dY\, ds
\approx \iint\limits_{\bigcup_{Q \in \sbf^*} U_Q^+}|\nabla_X v|^2 \, \tilde{\delta}(s,Y) \, dY\, ds
\\ \le \iint\limits_{\widetilde{\om}_{\sbf}^+}|\nabla_X v|^2 \, \tilde{\delta}(s,Y) \lesssim  \sigma(Q(\sbf^*)),
\end{multline}
where we used \eqref{omplusCME.eq} in the last inequality. This proves \eqref{unUQplusCME.eq},
and the proof of the theorem is complete.
\end{proof}

The proof of Theorem \ref{PURCMEMT2.thrm} is nearly identical, the only difference being that in this case one needs to take $\W(\om)$, a Whitney decomposition of $\om$, instead of $\W(E^c)$. Modulo the following remark we leave the details to the interested reader.

\begin{remark}
As in the elliptic setting, Theorem \ref{PURCMEMT2.thrm} does {\it not} require the corkscrew condition. On the other hand,
the converse of the {\em elliptic} version of Theorem \ref{PURCMEMT2.thrm} \cite{GMT,AGMT} requires the additional assumption of interior corkscrews. Note that when carrying out the proof of Theorem \ref{PURCMEMT2.thrm}, without the corkscrew assumption  it may be the case that the Whitney regions $U_Q$ are empty for some cubes $Q \in \mathbb{D}(\pom)$, but this does not affect the analysis above.
\end{remark}

\section{Further remarks}\label{FR.sect}

In this section we make some remarks concerning possible extensions
and consequences of Theorem \ref{PURCMEMT.thrm} and the constructions in Proposition \ref{stopdomainconst.prop}. These
extensions and consequences can be proved, or, we expect that they can be proved,
largely using the tools already developed in the elliptic setting. Again, we believe that the main novelty of this paper is the approximation scheme, that is, Proposition \ref{stopdomainconst.prop}.

The first observation is that solutions to the heat equation are (locally) smooth and that $t$-derivatives of solutions are, in fact, solutions. This allows one to produce a Caccioppoli-type inequality for the $t$-derivative which, in turn, allows one to improve the Carleson measure estimate in Theorems \ref{PURCMEMT.thrm} and \ref{PURCMEMT2.thrm} to one that includes the $t$-derivative. In particular, under the same hypotheses as Theorem \ref{PURCMEMT.thrm} the estimate
\begin{equation}\label{PURCMEMTimprove.thrm}
\sup_{(t,X) \in E, r > 0} r^{-n-1}
\iint\limits_{B((t,X),r)} \left(|\nabla u|^2 + \delta(s,Y)^2|\partial_s u|^2  \right) \delta(s,Y) \,  dY \, ds \le C \|u\|_{L^\infty(E^c)}^2,
\end{equation}
holds with a constant $C$ depending only dimension and the parabolic UR constants for $E$.

The second observation is that the proof of Theorem \ref{PURCMEMT.thrm} uses essentially three properties of $u$:
\begin{itemize}
\item[(i)]  $u \in L^\infty(E^c)$,
\item[(ii)] the Caccioppoli's inequality of Lemma \ref{cacc}, and
\item[(iii)] the local square function estimate stated in Lemma \ref{pullbackCME.lem}.
\end{itemize}
If one wants to extend the validity of Theorem \ref{PURCMEMT.thrm} to more general parabolic equations in divergence form, 
\[\mathcal{L} = \partial_t - \div_X A(t,X) \nabla_X,\]
where $A$ is a $n\times n$ uniformly elliptic matrix, then some regularity conditions on the coefficients need to be imposed in order to guarantee property (iii). A natural sufficient condition is the parabolic analogue of the ``Kenig-Pipher condition"\footnote{In fact, in \cite[Lemma A.2]{HL-JFA}, a slightly more general class of coefficients is permitted.}. More specifically, this means that $A$ satisfies
$$|\nabla_XA(s,Y)|\delta(s,Y),\ |\partial_t A(s,Y)| \delta^2(s,Y)\in L^\infty(\mathbb R^{n+1}\setminus E),$$
where $\delta(s,Y)=\dist ( (s,Y), E )$, and that there exists a constant $M$ such that
\begin{align}\label{measure2}
\sup_{(t,X) \in E, r > 0} r^{-n-1}
\iint\limits_{ B( (t, X ),r)} |\nabla_XA(s,Y)|^2 \delta(s,Y) \,  d Y d s&\le  M ,\notag\\
\sup_{(t,X) \in E, r > 0} r^{-n-1}
\iint\limits_{ B( (t, X ),r)} |\partial_t A(s,Y)|^2 \delta^3(s,Y) \,  d Y d s&\le M.
\end{align}
In particular, our results apply to this class of coefficients.

A final observation is that it seems likely that some form of $\epsilon$- approximability \cite{HMM1,HMM2} should hold in this parabolic setting along with the corresponding quantitative Fatou theorem \cite{BH2}. In fact, it may be more reasonable to use the dyadic constructions from \cite{HMM1} in proving these results. Indeed, our construction here would provide some of the necessary initial estimates (Theorem \ref{PURCMEMT.thrm}), but it seems easier to deduce (parabolic) BV estimates using dyadic cubes. We also mention that is natural to use the estimate \eqref{PURCMEMTimprove.thrm} and to
prove $\epsilon$-approximability via a Carleson measure estimate which includes the $t$-derivative of the
approximator. Note that this estimate was not included in \cite{RN1} and therefore the
proof used in \cite[Proposition 4.3]{RN1} is valid only if one works with a vertical version of the non-tangential counting function $\mathcal{N}$ (or by a spatially non-tangential version based on time-slice cones with
$t$ fixed), rather than a fully non-tangential version.

\appendix
\section{Proof of Lemma \ref{regL121.lem}}

Here we give the proof of Lemma \ref{regL121.lem}.
The idea is to follow Stein's construction of the regularized distance \cite[Chapter VI, \textsection 1 \&  \textsection 2]{St-SIO} and to combine this with ideas from some of the estimates produced in \cite{DS1}.
\begin{proof}[Proof of Lemma \ref{regL121.lem}] let $d := n +1$ and  $Z = \{(t,x'): h(t,x') = 0\}$. Then $Z$ is closed since $h$ is continuous. We set $H(t,x') = 0$ for all $(t,x') \in Z$.

We need to define $H$ off of $Z$, and following \cite{DS1},
we begin by producing a Whitney-type
decomposition\footnote{In contrast to the usual Whitney decomposition, in which $h$ is the distance to a
closed set,
the present version remains valid even in the case that $Z$ is empty.}
of $Z^c$ with respect to $h$. For each $(t,x') \in Z^c$, we let $I_{(t,x')}$ be the largest closed (parabolic) dyadic cube containing $(t,x')$ and satisfying
\[\diam(I_{(t,x')}) \le (1/20) \inf_{(\tau,z') \in I_{(t,x')}} h(\tau,z').\]
Recall that the diameter is defined with respect to parabolic metric. To see that such a cube exists, set $r = h(t,x')/2$ and note, as $h$ is Lip(1/2,1) with constant $1$, that $h(t,x') \ge r$ in $B((t,x'),r)$. Therefore,  any cube which contains $(t,x')$ and which has diameter less than $r/20$ is a `candidate' for $I_{(t,x')}$.
We conduct this construction for each $(t,x') \in Z^c$, and
we enumerate the resulting maximal cubes
(without repetition) as $\{I_i\}_{i \in \Lambda}$. We note that
\begin{equation}\label{Iidiamh.eq}
10 \diam(I_i) \le h(t,x') \le 60 \diam(I_i), \quad \forall (t,x') \in 10I_i,\end{equation}
where $\kappa I$ is the parabolic dilation of $I$ by a factor of $\kappa$. Indeed,  if $(t,x') \in 10I_i$ then $(t,x')$ is at most a (parabolic) distance of $10\diam(I_i)$ from a point in $I_i$ and hence, using the selection criterion for $I_i$ and the Lip(1/2,1)-condition for $h$,
\[h(t,x') \ge \min_{(\tau,z') \in I_i} h(\tau,z') - 10\diam(I_i) \ge 10\diam(I_i).\]
To prove the upper bound in \eqref{Iidiamh.eq}, we note that if $I$ is the parent of $I_i$, then $I$ fails the selection criteria. Hence there exists $(\tau, z') \in I$ such that $h(\tau,z') < 20 \diam(I) = 40\diam(I_i)$, and as $h$ is Lip(1/2,1) with constant $1$ and $\dist((\tau,z'), (t,x')) \le 20 \diam(I_i)$, it follows that
\[h(t,x') \le h(\tau,z') + 20 \diam(I_i) \le 60\diam(I_i).\]
Using \eqref{Iidiamh.eq} we have that 
\begin{equation}\label{compnear.eq}\frac{1}{6} \diam(I_j) \le \diam(I_i) \le 6 \diam(I_j),
\end{equation}
whenever $10I_i \cap 10_j \neq \emptyset$. By comparing volumes, it follows that $\{10I_j\}$ have bounded overlap, with a constant depending on dimension alone, that is,
\begin{equation}\label{whitbbdol.eq}
\sum_{i \in \Lambda} 1_{I_i}(t,x') \le N, \quad \forall (t,x') \in \rn
\end{equation}
where $N= N(n)$.

Let $Q_0 = \{(t,x') \in \rn: |x'|_\infty \le 1/2, |t| < 1/4\}$ be the unit parabolic cube in $\rn$. Let $\varphi \in C_0^\infty(3Q_0)$, $0 \le \varphi \le 1$, $\varphi \equiv 1$ on $2Q_0$. Clearly the Lip(1/2,1) constant of $\varphi$ is bounded. For each $i \in \Lambda$, let $(t_i, x'_i)$ be the center of $I_i$ and $\ell(I_i)$ be the parabolic side length of $I_i$, that is, $\ell(I_i) = r_i/2$ and
\[I_i = \{(t,x'): |x - x_i|_\infty < r_i, |t -t_i| < r_i^2\}. \]
Then for $i \in \Lambda_i$ we set
\[\varphi_i(t,x') = \varphi \left(\frac{t-t_i}{\ell(I_i)^2},\frac{x-x_i}{\ell(I_i)}\right). \]
Then $0 \le \varphi_i \le 1$, $\varphi_i$ is supported in $3I_i$, $\varphi_i \equiv 1$ on $2I_i$,
$\varphi_i$ is Lip(1/2,1) with constant less than $\ell(I_i)^{-1} \approx_n \diam(I_i)^{-1}$,
and (on all of $\rn$) it holds
\[\ell(I_i)^{2m}|\partial_t^m \varphi_i| + \ell(I_i)|\nabla_{x'} \varphi_i| \approx_{n,m} \diam(I_i)^{2m}|\partial_t^m \varphi_i| + \diam(I_i)|\nabla_{x'} \varphi_i| \lesssim \tilde{c}_{n,m}. \]
We now define
\[H(t,x') := \sum_{i \in \Lambda} \diam(I_i) \varphi_i(t,x').\]
Using \eqref{Iidiamh.eq} we see that $3I_i$ does not meet $Z$ for any $i \in \Lambda$, and hence $H(t,x') = 0$ for all $(t,x') \in Z$.
By construction, if $(t,x') \in Z^c$ then $(t,x') \in I_j$ for some $j \in \Lambda$ and as $\varphi_j(t,x') = 1$, using also \eqref{Iidiamh.eq}, we have that
\[H(t,x') \ge \diam(I_j) \ge (1/60) h(t,x').\] This proves the lower bound in (1). To prove the upper bound in (1) we have by \eqref{compnear.eq} and \eqref{whitbbdol.eq}
\[H(t,x') \le \sum_{i: 3I_i \cap 3I_j \neq \emptyset} \diam(Q_j) \le 6N\diam(Q_i) \le (3/5)Nh(t,x'),  \]
where we used \eqref{Iidiamh.eq} in the last line.
Having proved (1), we see that the proof of (2) is very similar. For instance, using the bounds for the $t$-derivatives of $\varphi_i$, if $(t,x') \in Z^c$ then $(t,x') \in I_j$ for some $j$, and hence
\begin{multline*}
|\partial_t^mH (t,x')| \le \tilde{c}_{n,m} \sum_{i: 3I_i \cap 3I_j \neq \emptyset} \diam(I_i) \diam(I_i)^{-2m} \\[4pt]
\lesssim CN\diam(I_j)^{-2m +1} \approx h(t,x')^{-2m + 1}.
\end{multline*}
The bound for $|\nabla^m_{x'}H|$ has a similar proof. Finally, to see that $H$ is Lip(1/2,1) we have
\[
\begin{split}|H(t,x') - H(s,y')|
&\le \sum_{i: (t,x') \in 3I_i} \diam(I_i)|\varphi_i(t,x') - \varphi(s,y')|
\\ & \quad + \sum_{i: (s,y') \in 3I_i} \diam(I_i)|\varphi_i(t,x') - \varphi(s,y')|
\\ &\le 2c'N[|t-s|^{1/2} + |x' - y'|],
\end{split}\]
where we used that $\varphi_i$ is a Lip(1/2,1) function with constant $c'\diam(I_i)^{-1}$.

Now we get to the heart of the matter, that is, proving the half-order in time regularity of $H$ (this part is not in Stein's book, but rather draws a great deal of inspiration from \cite{DS1}). By the results in \cite[pp. 370-373]{HLN-BP} it suffices to verify
the Carleson measure estimate
\begin{equation}\label{geocarlest.eq}\tilde{\nu}(s,y', \rho) \le c'_4 \rho^{n+1}, \quad \forall(s,y') \in \rn, \rho > 0
\end{equation}
where
\[\tilde{\nu}(s,y',\rho):= \int_0^\rho \iint_{B((s,y'), \rho)} \hat{\gamma}(\tau,z', r)^2 d\sigma(\tau, z') \, dr/r, \]
where $d\sigma(\tau,z') = \sqrt{1 + |\nabla_{z'}H(\tau,z')|} dz'dt$ and
\[\hat{\gamma}(\tau, z', r) := \inf_{L} \left[r^{-d}\iint_{B((\tau,z'), r)} \left(\frac{H(t,x') - L(x')}{r} \right)^2 \,d\sigma(t,x')\right]^{1/2}, \]
where the infimum is taken over all affine functions $L$ of $x'$ only, and we recall that $d=n+1$.

The idea behind the proof of the estimate \eqref{geocarlest.eq} is as follows. If the scale $r$ is large with respect to $h(\tau,z')$, then $H$ is well approximated by just the linear function $0$, and if the scale is small with respect to $h(\tau,z')$ then, necessarily, $h(\tau,z') >0$ and $H$ is flat (below the scale $r$) by the derivative estimates (2) and therefore we can approximate $H$ by its ``$x'$-tangent plane".

Now let us begin the proof of \eqref{geocarlest.eq}. Fix $(s,y') \in \rn$ and $\rho > 0$. Set
$h_\rho(t,x') := \min\{h(t,x')/60,\rho\}$,
and write
\begin{align*}
\tilde{\nu}(s,y',\rho)&= \int_0^\rho \iint_{B((s,y'), \rho)} \hat{\gamma}(\tau,z', r)^2 d\sigma(\tau, z') \, dr/r
\\ & =
\iint_{B((s,y'), \rho)}\int_{h_\rho(\tau,z')}^\rho \hat{\gamma}(\tau,z', r)^2 d\sigma(\tau, z') \, dr/r
\\ & \quad + \iint_{B((s,y'), \rho)}\int_0^{h_\rho(\tau,z')} \hat{\gamma}(\tau,z', r)^2 d\sigma(\tau, z') \, dr/r
\\ &= T_1 + T_2.
\end{align*}
Let us handle term $T_2$ first. For $(\tau,z')$ and $r$ in the domain of integration, $r > 0$, and $h(\tau,z')\geq 60r$. In particular, $(\tau,z') \in I_j$ for some $j \in \Lambda$ and
for all such $j$ it holds that $I_j \cap B((s,y'),\rho) \neq \emptyset$, and $r \le \diam(I_j)$
(by the right hand estimate in \eqref{Iidiamh.eq}). Thus, we have
\begin{equation}\label{T2bd.eq}T_2 \le \sum_{j \in \widetilde{\Lambda}}\iint_{I_j \cap B((s,y'),\rho) }\int_0^{\min\{\diam(I_j),\rho\}} \hat{\gamma}(\tau,z',r)^2\, \frac{dr}{r} \, d\sigma(\tau,z'),
\end{equation}
where $\widetilde{\Lambda} = \{j \in \Lambda: B((s,y'),\rho) \cap I_j \neq \emptyset\}$.
Fix $j \in \widetilde{\Lambda}$, $(\tau,z') \in I_j$ and $r \leq \diam(I_j)$.
Using the affine function
$$L_{(\tau,z')}(x') = H(\tau,z') + \nabla_{z'} H(\tau,z')\cdot (x' - z'),$$ we find by Taylor's theorem, Lemma \ref{regL121.lem} (2) (already proved above), and \eqref{Iidiamh.eq}, that
\begin{equation}\label{hatgamloc.eq}
    \begin{split}
    \hat{\gamma}(\tau,z',r)^2 &\le r^{-d} \iint_{B((\tau,z'),r)} \left(\frac{|H(t,x') - L_{(\tau,z')}(x')|}{r} \right)^2 \, d\sigma(t,x')
    \\ & \lesssim r^2 \sup_{B({\bf x}_{I_j},2\diam(I_j))} [|\partial_tH|^2 + |\nabla^2 H|]^2
    \\ & \lesssim r^2 \sup_{B({\bf x}_{I_j},2\diam(I_j))} h^{-2}(t,x') \lesssim r^2 \diam(I_j)^{-2}\\
    & \lesssim r^2 (\min\{\diam(I_j), \rho\})^{-2},
\end{split}
\end{equation}
where ${\bf x}_{I_j}= (t_{I_j}, x'_{I_j})$ is the center of $I_j$.
Using \eqref{hatgamloc.eq} and \eqref{T2bd.eq} we obtain
\begin{multline*}
    T_2 \lesssim \sum_{j \in \widetilde{\Lambda}}
    \iint_{I_j \cap B((s,y'),\,\rho) }\int_0^{\min\{\diam(I_j),\,\rho\}} r^2 (\min\{\diam(I_j), \rho\})^{-2}\, \frac{dr}{r} \,d\sigma(\tau,z')
    \\
    \lesssim  \sum_{j \in \widetilde{\Lambda}}\iint_{I_j \cap B((s,y'),\,\rho)} 1 d\sigma(\tau,z')
    \\ \lesssim \sum_{j \in \widetilde{\Lambda}} \sigma\left(I_j \cap B\big((s,y'),\rho\big)\right) \lesssim
    \sigma\left(B\big((s,y'),\rho)\big)\right) \lesssim \rho^d,
\end{multline*}
as desired.

Having obtained the desired bound for $T_2$,
we turn our attention to $T_1$. For $(\tau, z') \in \rn$ and $r > 0$ we set
\[\Lambda'(\tau,z',r) := \{i \in I_i \cap B((\tau,z'), r) \neq \emptyset\}.\]
Note that in term $T_1$, we have $r \in (h(\tau,z')/60 , \rho)$, so that $h(\hat{\tau},\hat{z}') < 61r$ for all $(\hat{\tau},\hat{z}') \in B((\tau,z'), r)$ because $h$ has Lip(1/2,1) constant $1$.
Hence, by \eqref{Iidiamh.eq}, we have $\diam(I_i) \le 7r \le 7\rho$ for all $i \in \Lambda'(\tau,z',r)$. In particular,
since $(\tau,z') \in B((s,y'), \rho)$,
\begin{equation}\label{A.7}
\bigcup_{i \in \Lambda'(\tau,z',r)} I_i \subset B\big((s,y'), 10\rho\big).
\end{equation}
For $(\tau,z') \in B((s,y'), \rho) $, with $r \in (h(\tau,z')/60 , \rho)$,
we plug $L = 0$ into the definition of $\hat{\gamma}$ and
use property (1) (which we have already established) along with \eqref{Iidiamh.eq} to see
\begin{multline}\hat{\gamma}(\tau, z',r)^2 \le r^{-d} \iint_{B((\tau,z'),r)} \left(\frac{H(t,x')}{r}\right)^2 \, d\sigma(t,x')
\\ \lesssim \sum_{i \in \Lambda'(\tau,z',r)} r^{-d} \iint_{I_i} \diam(I_i)^{2}r^{-2} \, d\sigma(t,x')
\\ \lesssim \sum_{i \in \Lambda'(\tau,z',r)} \left(\frac{\diam(I_i)}{r}\right)^{d+2} \lesssim \sum_{i \in \Lambda'(\tau,z',r)} \left(\frac{\diam(I_i)}{r}\right)^{d+1},
\end{multline}
where we used the fact that $\diam(I_i) \lesssim r$ in the estimate on the last line.  Thus, if we let
\[
\Lambda_0:= \{i \in \Lambda: I_i \subset B((s,y'), 10\rho),\ \diam(I_i) \le 7 \rho\},
\]
then using \eqref{A.7}, the definition of $ \Lambda'(\tau,z',r)$,
and again using the fact that $\diam(I_i) \leq 7r$ for $i \in \Lambda'(\tau,z',r)$, we obtain
\begin{align*}T_1 &\le \iint_{B((s,y'),\rho)} \int_{h(\tau, z')/60}^\rho \sum_{i \in \Lambda'(\tau,z',r)} \diam(I_i)^{d+1} \frac{dr}{r^{d+2}} \, d\sigma(\tau,z')
\\ & \lesssim \sum_{i \in \Lambda_0} \diam(I_i)^{d+1} \int_{\diam(I_i)/7}^\rho \int_{\dist((\tau,z'),I_i) < r} 1 \, d\sigma(\tau, z') \, \frac{dr}{r^{d+2}}
\\ & \lesssim \sum_{i \in \Lambda_0} \diam(I_i)^{d+1} \int_{\diam(I_i)/7}^\infty  
\frac{dr}{r^2}  \, \lesssim \, \sum_{i \in \Lambda_0} \diam(I_i)^{d}\, \lesssim\, \rho^d.
\end{align*}
This yields the desired Carleson measure estimate, and concludes the proof of the lemma.
\end{proof}

\end{document}